\documentclass[11pt]{article}
\usepackage[utf8]{inputenc}
\usepackage{centernot}
\usepackage[T1]{fontenc}
\usepackage{amsmath}
\usepackage{amsfonts}
\usepackage{amssymb}
\usepackage{amsthm}
\usepackage{mathtools}
\usepackage{hyperref}
\usepackage{url,graphicx,color}
\usepackage{mathscinet}
\usepackage[top=2cm, bottom=2cm, left=2.5cm, right=2.5cm]{geometry}

\newtheorem{theorem}{Theorem}[section]
\newtheorem{lemma}[theorem]{Lemma}

\newtheorem{corollary}[theorem]{Corollary}
\newtheorem{claim}{Claim}

\theoremstyle{definition}
\newtheorem{defn}{Definition}[section]

\theoremstyle{remark}
\newtheorem{remark}{Remark}[section]

\newcommand\Sgn{\operatorname{Sgn}}

\newcommand\sgn{\operatorname{sgn}}
\newcommand\Sp{\operatorname{Sp}}
\newcommand\Id{\operatorname{Id}}

\newcommand\Span{\operatorname{Span}}

\newcommand\ind{\operatorname{ind}}

\allowdisplaybreaks

\begin{document}
\hypersetup{pageanchor=false}
\title{On bifurcation of eigenvalues along convex symplectic paths}
\author{Yinshan Chang$^{1}$\thanks{Supported by the Fundamental Research Funds for the Central Universities (No. YJ201661). E-mail: ychang@scu.edu.cn}, Yiming Long$^2$\thanks{Partially supported by NSFC Grants 11131004 and 11671215, LPMC of Ministry of Education of China, Nankai University, and the Beijing Center for Mathematics and Information Interdisciplinary Sciences at Capital Normal University. E-mail: longym@nankai.edu.cn},
Jian Wang$^3$\thanks{Supported by China Postdoctoral Science Foundation (No. 2013T60251) and Natural Science Foundation of China (No. 11401320 and No. 11571188). E-mail: wangjian@nankai.edu.cn} \\$^1$Department of Mathematics, University of Sichuan, Chengdu 610065\\$^2$Chern Institute of Mathematics and LPMC, Nankai University, Tianjin 300071 \\$^{3}$Chern Institute of Mathematics, Nankai University, Tianjin 300071
}

\date{}
\maketitle

\begin{abstract}
 We consider a continuously differentiable curve $t\mapsto \gamma(t)$ in the space of $2n\times 2n$ real symplectic matrices, which is the solution of the following ODE:
 \begin{equation*}
 \frac{\mathrm{d}\gamma}{\mathrm{d}t}(t)=J_{2n}A(t)\gamma(t), \gamma(0)\in\Sp(2n,\mathbb{R}),
\end{equation*}
where $J=J_{2n}\overset{\text{def}}{=}\begin{bmatrix}0 & \Id_n\\-\Id_n & 0\end{bmatrix}$ and $A:t\mapsto A(t)$ is a continuous in the space of $2n\times2n$ real matrices which are symmetric. Under certain convexity assumption (which includes the particular case that $A(t)$ is strictly positive definite for all $t\in\mathbb{R}$), we investigate the dynamics of the eigenvalues of $\gamma(t)$ when $t$ varies, which are closely related to the stability of such Hamiltonian dynamical systems. We rigorously prove the qualitative behavior of the branching of eigenvalues and explicitly give the first order asymptotics of the eigenvalues. This generalizes classical Krein-Lyubarskii theorem on the analytic bifurcation of the Floquet multipliers under a linear perturbation of the Hamiltonian. As a corollary, we give a rigorous proof of the following statement of Ekeland: $\{t\in\mathbb{R}:\gamma(t)\text{ has a Krein indefinite eigenvalue of modulus }1\}$ is a discrete set.
\end{abstract}

\section{Introduction}
\subsection{The introduction of the model and the main assumption}
We consider linearized Hamiltonian equations in $\mathbb{R}^{2n}$ of the following type
\begin{equation}\label{eq: ham system R modified initial cont pos}
 \frac{\mathrm{d}\gamma}{\mathrm{d}t}(t)=J_{2n}A(t)\gamma(t), \gamma(0)\in\Sp(2n,\mathbb{R}),
\end{equation}
where $J=J_{2n}\overset{\text{def}}{=}\begin{bmatrix}0 & \Id_n\\-\Id_n & 0\end{bmatrix}$ and $A:t\mapsto A(t)$ is a continuous periodic curve in the space of $2n\times2n$ real matrices which are symmetric with the periodicity $T$. The unique solution is a curve in the space of real symplectic matrices such that
\begin{equation}\label{eq: fundamental solution certain multiplicative stationary increment}
 \gamma(t+T)\gamma(T)^{-1}=\gamma(t)\gamma(0)^{-1}.
\end{equation}
The system \eqref{eq: ham system R modified initial cont pos} arises naturally from perturbations of linearized Hamiltonian equations. Indeed, let $\varepsilon\in\mathbb{R}$ be a \emph{real} perturbation parameter. Consider
\begin{equation}\label{eq: perturbed ham system R}
 \frac{\partial\gamma}{\partial t}(t,\varepsilon)=J_{2n}A(t,\varepsilon)\gamma(t,\varepsilon),\quad \gamma(0,\varepsilon)=\Id_{2n},
\end{equation}
where
$t\mapsto A(t,\varepsilon)$ is a locally integrable periodic curve in the space of $2n\times2n$ real matrices which are symmetric and periodic with the periodicity $T$. Moreover, we assume that $\varepsilon\mapsto(A(t,\varepsilon),t\in[0,T])$ is a \emph{continuously Fr\'{e}chet-differentiable} curve in $L^{1}[0,T]$. Then, for fixed $T$, as $\varepsilon$ varies, the endpoint matrix $\gamma(T,\varepsilon)$ is a $C^{1}$-curve satisfying \eqref{eq: ham system R modified initial cont pos}. More precisely,
\begin{equation}\label{eq: equation for partial partial gamma T epsilon}
 \frac{\mathrm{\partial}}{\mathrm{\partial}\varepsilon}\gamma(T,\varepsilon)=\gamma(T,\varepsilon)J_{2n}C(T,\varepsilon)=J_{2n}B(T,\varepsilon)\gamma(T,\varepsilon),
\end{equation}
where
\begin{equation}\label{defn: CTepsilon}
 C(T,\varepsilon)=-\gamma(T,\varepsilon)^{T}J_{2n}\frac{\partial}{\partial\varepsilon}\gamma(T,\varepsilon)=\int_{0}^{T}\gamma(t,\varepsilon)^{T}\frac{\partial}{\partial\varepsilon}A(t,\varepsilon)\gamma(t,\varepsilon)\,\mathrm{d}t
\end{equation}
and $B(T,\varepsilon)=(\gamma(T,\varepsilon)^{-1})^{T}C(T,\varepsilon)\gamma(T,\varepsilon)^{-1}$, where the superscript ``$T$'' denotes the transpose of matrices. Note that both $C$ and $B$ are symmetric real matrices and they are continuous in $\varepsilon$. We refer to Subsection~\ref{subsect: lem  expression of C} for the second inequality in \eqref{defn: CTepsilon}.

Let us go back to the system \eqref{eq: ham system R modified initial cont pos} and recall that a matrix $\gamma$ is called \emph{stable} if $\sup_{n\in\mathbb{Z}}||\gamma^n||<\infty$. We say that the system \eqref{eq: ham system R modified initial cont pos} is \emph{stable} if the matrix $\gamma(T)$ is stable. By \eqref{eq: fundamental solution certain multiplicative stationary increment}, we have that $\sup_{t\in\mathbb{R}}||\gamma(t)||<\infty$ if $\gamma(T)$ is stable. A symplectic matrix $\gamma$ is called \emph{strongly stable} if there exists a neighborhood of $\gamma$ in the space of symplectic matrix containing only stable symplectic matrices. We say that the system \eqref{eq: ham system R modified initial cont pos} is strongly stable if $\gamma(T)$ is strongly stable as a symplectic matrix. In this case, when the system \eqref{eq: ham system R modified initial cont pos} is slightly perturbed, it is still a stable system. The picture is not clear in general if we perturb a stable but not strongly stable system.

The stability is closely related to the eigenvalues of a symplectic matrix. We give a brief explanation in the following. For more details, please refer to \cite[Sections~1.1 and 1.2]{EkelandMR1051888}. The eigenvalues of a sympletic matrix come in $4$-tuple like $\{\lambda,\lambda^{-1},\bar{\lambda},\bar{\lambda}^{-1}\}$ and hence it is stable iff it is diagonalizable and all its eigenvalues stay on the unit circle $U\subset\mathbb{C}$. The characterization of strong stability was firstly formulated by Krein \cite{KreinMR0036379, KreinMR0043980}, and later independently by Moser \cite{MoserMR0096872}, as stated in the following: a symplectic matrix $\gamma$ is strongly stable iff it is stable and all its eigenvalues are \emph{Krein definite}. To be more precise, let $G=-\sqrt{-1}J$ be the Krein form which gives an inner product on $\mathbb{C}^{2n}$ via
\begin{equation}\label{eq: defn of G}
(x,y)_{G}=\sqrt{-1}\left\{\sum_{k=1}^{n}(x_k\bar{y}_{n+k}-x_{n+k}\bar{y}_{k})\right\}.
\end{equation}
Then, an eigenvalue $\lambda\in U$ is said to be \emph{Krein positive (resp. negative) definite} if the bilinear form $(x,y)\mapsto (x,y)_{G}$ is positive (resp. negative) definite on the invariant space $E_{\lambda}$ associated with the eigenvalue $\lambda$, see Subsection~\ref{subsect: notation} for the definition of $E_{\lambda}$. It is called \emph{Krein indefinite} if the bilinear form $(x,y)\mapsto (x,y)_{G}$ is indefinite on $E_{\lambda}$.

Under the convexity assumption that \emph{$A(t)$ is strictly positive definite for all $t\in\mathbb{R}$}, Ekeland \cite[Section~1.3]{EkelandMR1051888} has investigated the system~\eqref{eq: ham system R modified initial cont pos} when $\gamma(0)=\Id$. Among various results, Ekeland has claimed that the following set is isolated:
\begin{equation}\label{eq: defn D ham system}
 D\overset{\mathrm{def}}{=}\{t:\gamma(t)\text{ has a Krein indefinite eigenvalue on }U\},
\end{equation}
see \cite[Proposition~4, Section~1.3]{EkelandMR1051888}. However, later, in \cite[Erratum]{EkelandMR1051888}, Ekeland wrote that ``The proof of Proposition~4 (and probably the proposition itself) is wrong'', and he proved a weaker statement for continuous $t\mapsto A(t)$:
$D$ is a finite union of isolated sets $D_{m}$, where
\[D_{m}\overset{\mathrm{def}}{=}\left\{t\in D: \begin{array}{l}
\text{all Krein indefinite eigenvalues of }\gamma(T)\text{ have algebraic multiplicity}\\
\text{at most }m\text{ and one of them having exactly multiplicity }m
\end{array}\right\},\]
see Pages $1$ and $2$ in \cite[Erratum]{EkelandMR1051888}.

We prove that the original statement of Ekeland is still correct under the following weaker assumption on $A$:
\begin{equation}\label{defn: weaker definiteness assumption}
 A(t)\text{ is strictly positive definite on }\ker(\omega\cdot\Id-\gamma(t))\text{ for all }t\in\mathbb{R}\text{ and }\omega\in U.
\end{equation}
\begin{theorem}\label{thm: ekeland conj}
 For the system \eqref{eq: ham system R modified initial cont pos} with continuous (but not necessarily periodic) $t\mapsto A(t)$, if \eqref{defn: weaker definiteness assumption} holds, then the set $D$ defined in \eqref{eq: defn D ham system} is discrete.
\end{theorem}
To understand the system \eqref{eq: ham system R modified initial cont pos} and prove Theorem~\ref{thm: ekeland conj}, we need to study the dynamics of the eigenvalues and the associated Krein forms as $t$ varies. There is a rather complete answer for linear perturbations of Hamiltonians of Krein positive type. To be more precise, consider the endpoint matrix $\gamma(T,\varepsilon)$ of the system \eqref{eq: perturbed ham system R} with $\varepsilon\in\mathbb{C}$ and $A(t,\varepsilon)=H(t)+\varepsilon Q(t)$, where $H(t)$ and $Q(t)$ are both $2n\times 2n$ Hermitian matrices. The perturbation is said to be of Krein positive type if $Q$ is non-negative definite and for all $\omega\in U$, there is no solution of the following equations in $\mathbb{C}^{2n}$:
\begin{equation*}
 \frac{\mathrm{d}}{\mathrm{d}t}x(t)=JH(t)x(t)\text{ and }Q(t)x(t)=0\text{ a.e., }x(T)=\omega x(0).
\end{equation*}
Although $\varepsilon$ is complex, by similar arguments, we see that \eqref{eq: equation for partial partial gamma T epsilon} and \eqref{defn: CTepsilon} also hold. And the condition of Krein positive type perturbation is precisely the condition \eqref{defn: weaker definiteness assumption} by replacing $t$ by $\varepsilon$, $A(t)$ by $B(T,\varepsilon)$ and $\gamma(t)$ by $\gamma(T,\varepsilon)$. In this special case, Krein-Lyubarskii theorem \cite{KreinLyubarskiiMR0142832} asserts the analytic properties of the eigenvalues and the eigenvectors.
\begin{theorem}[Krein-Lyubarski]
 Consider the system \eqref{eq: perturbed ham system R} with $A(t,\varepsilon)=H(t)+\varepsilon Q(t)$ and assume the perturbation is of Krein positive type. Suppose that $\varepsilon_0\in\mathbb{R}$ and that $\lambda_0\in U$ is an eigenvalue of $\gamma(T,\varepsilon_0)$. Then, as $\varepsilon$ varies from $\varepsilon_0$, $\lambda_0$ continuously branches into $\kappa$-many eigenvalues, where $\kappa=\dim\ker(\lambda_0\cdot\Id-\gamma(T,\varepsilon_0))^{2n}$ is the algebraic multiplicity of $\lambda_0$. These eigenvalues are grouped into $m$-groups, where $m=\dim\ker(\lambda_0\cdot\Id-\gamma(T,\varepsilon_0))$ is the geometric multiplicity of $\lambda_0$. Each group of eigenvalues forms a multi-valued analytic function with Puiseux expansions: for $i=1,\ldots,m$,
 \begin{equation*}
  \lambda_i(\varepsilon)-\lambda_0=\sum_{k=1}^{\infty}c_{i,k}(\varepsilon-\varepsilon_0)^{\frac{k}{j_i}},
 \end{equation*}
 where the numbers $j_1,\ldots,j_m$ are the sizes of Jordan blocks associated with the eigenvalue $\lambda_0$. In each of the expansions, the first coefficient $c_{i,1}$ ($i=1,\ldots,m$) is non-zero. For each group of eigenvalues $\lambda_i(\varepsilon)$, the eigenvalues branch from $\lambda_0$ with tangents as $\varepsilon\in\mathbb{R}$ increases from $\varepsilon_0$. These tangents form a $j_i$-star with the same angle between consecutive tangents. As $\varepsilon$ decreases from $\varepsilon_0$, the trajectories of eigenvalues also form another $j_i$-star. These two stars differ from each other by a rotation of $\frac{\pi}{j_i}$ radians. Among these $2j_i$ many tangents, exactly two are tangential to the circle at $\lambda_0$. If the trajectory of an eigenvalue branching from $\lambda_0$ is tangential to the circle $U$ at $\lambda_0$ as $\varepsilon$ varies, then that eigenvalue is Krein definite and moves on the circle $U$ in a definite direction for $\varepsilon$ sufficiently close to $\varepsilon_0$.
\end{theorem}
 See Figure~\ref{fig: bif star} for illustrations of a $2$-star and a $3$-star. The arrows indicate moving directions of the eigenvalues as $\varepsilon$ increases.
 \begin{figure}
 \centering
 \label{fig: bif star}
 \includegraphics[width=0.4\textwidth]{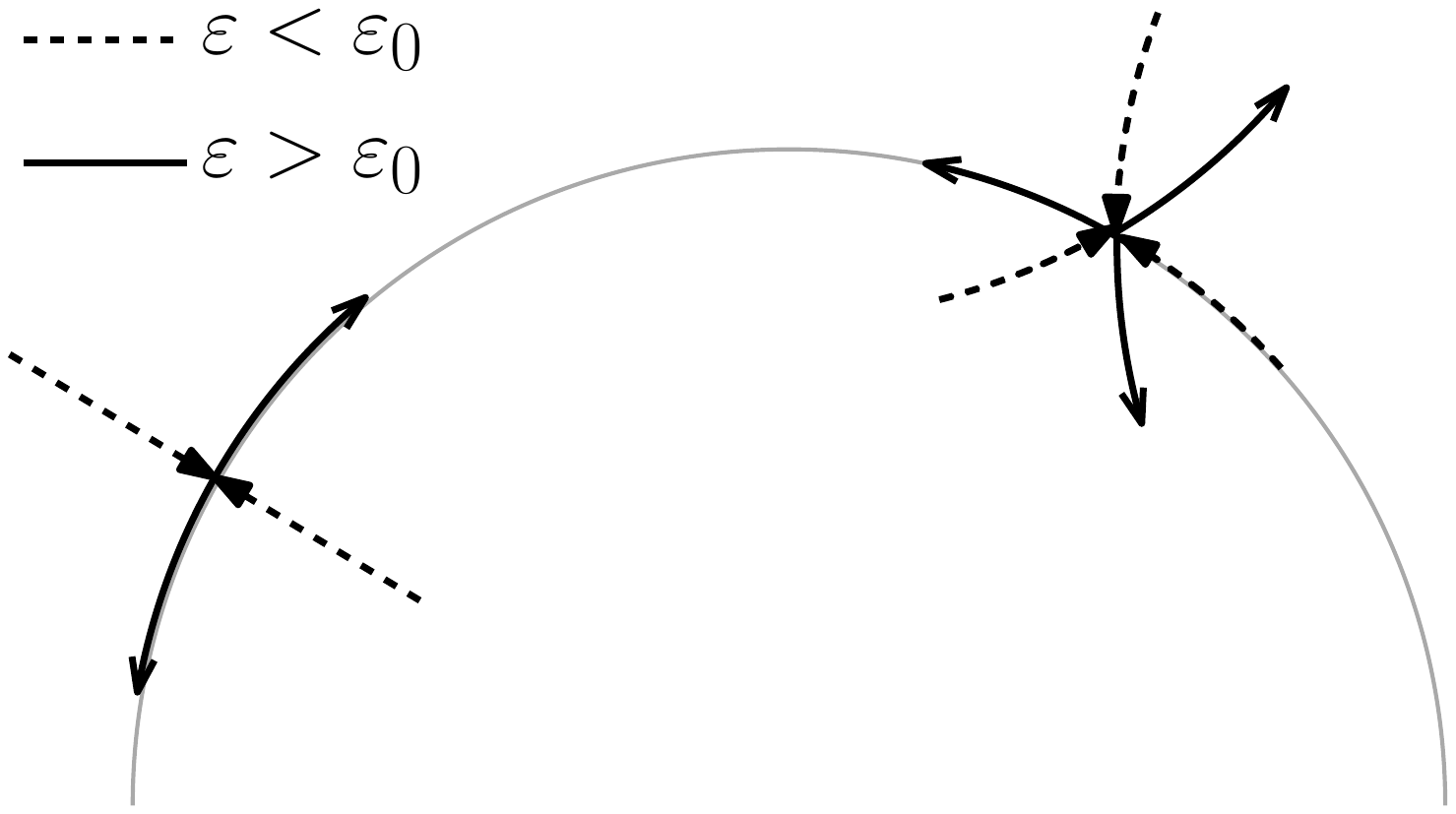}
  \caption{Bifurcation of eigenvalues}
 \end{figure}
\begin{remark}
 The eigenvectors also admit expansions in Puiseux seris as the eigenvalues, see \cite{YakubovichStarzhinskiiMR0364740}.

 In the proof of the above theorem, they also gave a recursive way to calculate $c_{i,1}$ via the matrix $Q$ and the generalized eigenvectors of $\gamma(T,\varepsilon)$ associated with $\lambda_0$. In the special case that $m=1$ or $j_1=\cdots=j_m=1$, such an expression were obtained earlier by Gelfand and Lidskii \cite{GelfandLidskiiMR0091390}. It also implies that Krein positive (resp. negative) definite eigenvalues move counter-clockwise (resp. clockwise) on the circle as the perturbation parameter $\varepsilon$ increases along the real axis. If several eigenvalues collide on the circle from $U^c$, then, necessarily, a Krein indefinite eigenvalue with non-trivial Jordan blocks (Jordan blocks of size $\geq 2$) is created. When several eigenvalues of different Krein types meet at $\lambda_0$ on the circle, they will continue their movement along the circle iff the geometric multiplicity of $\lambda_0$ equals to its algebraic multiplicity.

 Particularly, Krein-Lyubarskii theorem implies Theorem~\ref{thm: ekeland conj} for the curve $\varepsilon\mapsto\gamma(T,\varepsilon)$ given by \eqref{eq: perturbed ham system R} when $A(t,\varepsilon)=H(t)+\varepsilon Q(t)$. Indeed, by Krein-Lyubarskii theorem, for all $\varepsilon_0\in\mathbb{R}$, there exists $\delta=\delta(\varepsilon_0)>0$ such that for $\varepsilon\in(\varepsilon_0-\delta,\varepsilon_0)\cup(\varepsilon_0,\varepsilon_0+\delta)$, the eigenvalues on the circle are Krein definite.
\end{remark}
We would like to obtain a $C^{1}$-version of Krein-Lyubarskii theorem for the system \eqref{eq: ham system R modified initial cont pos} and prove that $D$ is isolated. For general $C^{1}$-perturbations, the eigenvalues and eigenvectors are no longer multi-valued analytic functions. Instead, we aim to give the first order asymptotic of the deviation of eigenvalues and to verify similar qualitative behavior of the dynamics of eigenvalues.

The argument of Krein and Lyubarskii doesn't directly apply. Their proof relies on a key lemma, which interprets the perturbation parameter $\varepsilon$ as an eigenvalue of certain self-adjoint integral operator depending on $\omega\in U$, see the lemma in \cite[Section~1]{KreinLyubarskiiMR0142832}. In this step, the linearity of the perturbation $\varepsilon\mapsto H(t)+\varepsilon Q(t)$ is crucially used. Beyond the scope of linear perturbations of Hamiltonians, if we assume the analyticity of $\varepsilon\mapsto A(t,\varepsilon)$ and follow their idea, we may encounter self-adjoint integral operators $G(\varepsilon,\omega)$ depending on two parameters $\varepsilon\in\mathbb{R}$ and $\omega\in U$. We have to show that $\{(\omega,\varepsilon):0\text{ is an eigenvalue of }G(\varepsilon,\omega)\}$ is actually the graph of an analytic function in $\omega$, which we regard as a difficult question in general. Besides, more seriously, their argument depends heavily on the analyticity of the system. This rules out the possibility of studying $C^{1}$-perturbations of the system by following their argument.

Ekeland has investigated the system \eqref{eq: ham system R modified initial cont pos} when $\gamma(0)=\Id$, $t\mapsto A(t)$ is continuous and $A(t)$ is strictly positive definite symmetric matrices for all $t$, see \cite{EkelandMR1051888}. It was proved that the moving direction of a Krein definite eigenvalue is determined by its Krein type: as $t$ increases a bit, the Krein positive (resp. negative) definite eigenvalues of $\gamma(t)$ move counter-clockwise (resp. clockwise). Krein indefinite eigenvalues appear when Krein positive definite eigenvalues meet Krein negative definite eigenvalues. He has also described the branching of a Krein indefinite eigenvalue of $\gamma(t)$ when $t$ varies from $t_0$ if $\gamma(t_0)=\Id$: if $\gamma(t_0)=\Id$, then there exists $\varepsilon_0>0$ such that for $t\in(t_0,t_0+\varepsilon_0]$ (resp. $t\in[t_0-\varepsilon_0,t_0)$), the eigenvalues of $\gamma(t)$ are all located on the unit circle, the eigenvalues on the upper semi circle are all Krein positive (resp. negative) definite and move counter-clockwise (resp. clockwise), while the eigenvalues on the lower part are all Krein negative (resp. positive) definite and move clockwise (resp. counter-clockwise). We remark that the condition $\gamma(0)=\Id$ is not essential in the above results of Ekeland. It suffices to have $\gamma(0)\in\Sp(2n,\mathbb{R})$.

In the same book, Ekeland has commented that the spirit of the branching mechanism of a Krein indefinite eigenvalue should be the same as in the special case of linear perturbations of Hamiltonians studied by Krein and Lyubarskii. However, to the best of our knowledge, there is no rigorous proof in general. Recently, when the Krein indefinite eigenvalue has algebraic multiplicity $2$ and geometric multiplicity $1$, Kuwamura and Yanagida \cite[Theorem~3.2]{KuwamuraYanagidaMR2271499} give a simple and elegant formula on the derivative of the mean of bifurcated eigenvalues, which holds \emph{without} the assumption \eqref{defn: weaker definiteness assumption}. In our opinion, under the assumption \eqref{defn: weaker definiteness assumption}, the first order terms of the pair of bifurcated eigenvalues cancel with each other and the second order terms of the pair is the same. And their formula is actually an expression for the second order term.

In the present paper, we focus on the first order term under the assumption \eqref{defn: weaker definiteness assumption} (but without any restriction on the multiplicities of the eigenvalues). Naturally, to study the branching of Krein indefinite eigenvalues $e^{\sqrt{-1}\theta_0}\in U$ of $\gamma(0)$, we need information on the Jordan blocks associated with $e^{\sqrt{-1}\theta_0}$. We need to introduce several notations for a precise statement of our $C^1$-version of Krein-Lyubarskii theorem. Note that there is a basis $\{\xi_{i,j}\}_{i=1,\ldots,m;j=1,\ldots,j_i}$ of the invariant space $E_{e^{\sqrt{-1}\theta_0}}(\gamma(0))=\ker(e^{\sqrt{-1}\theta_0}\cdot\Id-\gamma(0))^{2n}$ associated with the eigenvalue $e^{\sqrt{-1}\theta_0}$ of the matrix $\gamma(0)$ such that $m$ is the number of the Jordan blocks associated with the eigenvalue $e^{\sqrt{-1}\theta_0}$ of the matrix $\gamma(0)$, $j_1\geq j_2\geq \cdots \geq j_m\geq 1$ are the sizes of the Jordan blocks and $\{\xi_{i,j}\}_{i,j}$ are the corresponding eigenvectors, i.e., for $i=1,\ldots,m$ and $j=1,\ldots,j_i$, we have that
\begin{equation}\label{eq: gammaxi=lambda0xi-xi}
 \gamma(0)\xi_{i,j}=e^{\sqrt{-1}\theta_0}\xi_{i,j}-\xi_{i,j-1}\text{ for }j=1,\ldots,j_i,
\end{equation}
with $\xi_{i,0}=0$ and that
\begin{equation}\label{eq: xiij basis}
 \{\xi_{i,j}\}_{i=1,\ldots,m;j=1,\ldots,j_i}\text{ is a linear basis of }E_{e^{\sqrt{-1}\theta_0}}(\gamma(0)).
\end{equation}
Note that $j_1\geq \cdots\geq j_m\geq 1$ is not necessarily strictly decreasing. We break the sequence $\{j_{i}\}_{i}$ at the position where a strict decrease occurs. So, there are integers $s\geq 1$, $m_1,\ldots,m_s\geq 1$, $n_1>n_2>\cdots n_s\geq 1$ such that for $\ell=1,\ldots,s$, the integer number $n_{\ell}$ is the $\ell$-th largest size of Jordan blocks (in the strict sense) and there are exactly $m_{\ell}$ many blocks with the same size $n_{\ell}$. Hence, the total number of blocks $m=\sum_{\ell=1}^{s}m_{\ell}$ and for $\ell=1,\ldots,s$, we have that
\begin{equation}\label{eq: xiij regroup}
 j_i=n_{\ell},\text{ for }\sum_{1\leq k<\ell}m_{k}+1\leq i\leq\sum_{1\leq k\leq \ell}m_{k}.
\end{equation}
Sometimes, it is convenient\footnote{As we shall see in \eqref{eq: defn X xiJxi}, it helps to simplify the definition of $X$. Besides, the equation \eqref{eq: xiij G product binom like} is simpler in terms of $\{\eta_{i,j}\}$: $(\eta_{i,j},\eta_{i',j'})_{G}=\sqrt{-1}(\eta_{i,j+1},\eta_{i',j'})_{G}-\sqrt{-1}(\eta_{i,j},\eta_{i',j'+1})_{G}$.} to use the following sequence of vectors $\{\eta_{i,j}\}_{i=1,\ldots,m;j=1,\ldots,j_i}$ instead of $\{\xi_{i,j}\}_{i=1,\ldots,m;j=1,\ldots,j_i}$, where
\begin{equation}\label{eq: definition of etaij}
 \eta_{i,j}\overset{\mathrm{def}}{=}\left(-\sqrt{-1}e^{\sqrt{-1}\theta_0}\right)^{j}\xi_{i,j},
\end{equation}
for $i=1,\ldots,m$ and $j=1,\ldots,j_i$.
We need to introduce more notations to present our results. Define an $m\times m$ square matrix $S$, which represents the metric $\langle A(0)\cdot,\cdot\rangle$ on the space of eigenvectors associated with $e^{\sqrt{-1}\theta_0}$:
 \begin{equation}\label{eq: defn S Cxixi}
  S_{i,i'}=\langle A(0)\eta_{i,1},\eta_{i',1}\rangle=\langle A(0)\xi_{i,1},\xi_{i',1}\rangle,\quad i,i'=1,\ldots,m.
 \end{equation}
 We define an $m\times m$ square matrix $X$ by
 \begin{equation}\label{eq: defn X xiJxi}
  X_{i,i'}=(\eta_{i,j_i},\eta_{i',1})_{G}=(-1)^{j_i-1}\sqrt{-1}^{j_i}e^{(j_i-1)\sqrt{-1}\theta_0}\langle\xi_{i,j_i},J_{2n}\xi_{i',1}\rangle,\quad i,i'=1,\ldots,m.
 \end{equation}
 We write $S$ and $X$ in blocks as follows:
 \begin{equation}\label{eq: SX blocks}
  S=\begin{bmatrix}
  S^{(1,1)} & \cdots & S^{(1,s)}\\
  \vdots & \ddots & \vdots\\
  S^{(s,1)} & \cdots & S^{(s,s)}
  \end{bmatrix}\text{ and }X=\begin{bmatrix}
  X^{(1,1)} & \cdots & X^{(1,s)}\\
  \vdots & \ddots & \vdots\\
  X^{(s,1)} & \cdots & X^{(s,s)}
  \end{bmatrix},
 \end{equation}
 where $S^{(\ell,\ell')}$ and $X^{(\ell,\ell')}$ are $m_{\ell}\times m_{\ell'}$ matrices for $\ell,\ell'=1,\ldots,s$. A nice feature of $X$ is that $X$ is upper triangular in block sense and the diagonal blocks are Hermitian, see Corollary~\ref{cor: Xellell hermitian non-degenerate xell1ell2 vanishes}.

\begin{theorem}\label{thm: krein lyubarskii c1}
 Consider the system \eqref{eq: ham system R modified initial cont pos} and assume \eqref{defn: weaker definiteness assumption}. Suppose that $e^{\sqrt{-1}\theta_0}$ ($\theta_0\in\mathbb{R}$) is a Krein indefinite eigenvalue of $\gamma(0)$. Recall the notations introduced in \eqref{eq: gammaxi=lambda0xi-xi}, \eqref{eq: xiij basis}, \eqref{eq: xiij regroup}, \eqref{eq: defn S Cxixi}, \eqref{eq: defn X xiJxi} and \eqref{eq: SX blocks}.
 \begin{itemize}
 \item[a)]
 As $t$ varies from $0$, the eigenvalue $e^{\sqrt{-1}\theta_0}$ branches continuously into $\sum_{\ell=1}^{s}m_{\ell}n_{\ell}$ many eigenvalues with multiplicities, namely $\{\lambda_{\ell,p,q}(t)\}_{\ell=1,\ldots,s;p=1,\ldots,m_{\ell};q=1,\ldots,n_{\ell}}$.

 For $\ell=1,\ldots,s$, reordering $\{\lambda_{\ell,p,q}(t)\}_{q=1,\ldots,n_{\ell}}$ if necessary, we have that
 \begin{equation}\label{eq: asymp eigenvalues}
  \frac{\lambda_{\ell,p,q}(t)-e^{\sqrt{-1}\theta_0}}{\sqrt{-1}e^{\sqrt{-1}\theta_0}}\overset{t\to 0}{\sim}\left\{
  \begin{array}{ll}
  \sgn(t a_{\ell,p}) |a_{\ell,p}t|^{\frac{1}{n_{\ell}}}e^{\frac{2\pi}{n_{\ell}}\sqrt{-1}(q-1)} & \text{ if }n_{\ell}\text{ is odd},\\
  |a_{\ell,p}t|^{\frac{1}{n_{\ell}}}e^{\frac{2\pi}{n_{\ell}}\sqrt{-1}(q-1)}e^{\frac{\pi}{2n_{\ell}}\sqrt{-1}(1-\sgn(t a_{\ell,p}))} & \text{ if }n_{\ell}\text{ is even},
  \end{array}
  \right.
 \end{equation}
 where $(a_{\ell,p})_{p=1,\ldots,m_{\ell}}$ are non-zero real numbers and they are exactly the roots with multiplicities of the following polynomial in $z$
 \begin{equation}\label{eq: defn of aellp}
  \det\begin{bmatrix}
  S^{(1,1)} & \cdots & S^{(1,\ell-1)} & S^{(1,\ell)}\\
  \vdots & \ddots & \vdots & \vdots\\
  S^{(\ell-1,1)} & \cdots & S^{(\ell-1,\ell-1)} & S^{(\ell-1,\ell)}\\
  S^{(\ell,1)} & \cdots & S^{(\ell,\ell-1)} & S^{(\ell,\ell)}-zX^{(\ell,\ell)}
  \end{bmatrix}.
 \end{equation}
 \item[b)]
 There exists $\delta_0>0$ such that for $t\in(-\delta_0,0)\cup(0,\delta_0)$, $\ell=1,\ldots,s$ and $p=1,\ldots,m_{\ell}$, $(\lambda_{\ell,p,q}(t))_{q=1,\ldots,n_{\ell}}$ have different behaviors depending on the parity of $n_{\ell}$ and the sign of $t a_{\ell,p}$: if $n_{\ell}$ is odd, then $(\lambda_{\ell,p,q}(t))_{q=2,\ldots,n_{\ell}}$ stay outside of the unit circle $U$, and $\lambda_{\ell,p,1}$ is Krein positive definite on $U$ (resp. Krein negative definite) if $t a_{\ell,p}>0$ (resp. $t a_{\ell,p}<0$). If $n_{\ell}$ is even and $t a_{\ell,p}<0$, then $(\lambda_{\ell,p,q}(t))_{q=1,\ldots,n_{\ell}}$ stay outside of the unit circle $U$; if $n_{\ell}$ is even and $t a_{\ell,p}>0$, then $\lambda_{\ell,p,1}(t)\in U$ is Krein positive definite, $\lambda_{\ell,p,n_{\ell}/2+1}(t)\in U$ is Krein negative definite, and the other $\lambda_{\ell,p,q}(t)$ stay outside of $U$.
 \end{itemize}
\end{theorem}
\begin{remark}\label{rem: quanlitative A independence}
Note that $X^{(\ell,\ell)}$ is Hermitian and non-degenerate, see Corollary~\ref{cor: Xellell hermitian non-degenerate xell1ell2 vanishes}. By Sylvester's law of inertia, the number $\sharp \{p=1,\ldots,m_{\ell}:a_{\ell,p}>0\}$ equals the positive index of inertia of $X^{(\ell,\ell)}$. Hence, the instant moving directions of the eigenvalues (when $t$ increases (or decreases) from $0$), is purely determined by $\gamma(0)$ under the assumption \eqref{defn: weaker definiteness assumption}. When $t$ is sufficiently close to $0$, the number of the Krein positive (or negative) definite eigenvalues depends only on $\gamma(0)$.
\end{remark}
\begin{remark}
 If we replace ``positive definiteness'' by ``negative definiteness'' in \eqref{defn: weaker definiteness assumption}, i.e.,
  \begin{equation}\label{defn: weaker definiteness assumption 2}
 A(t)\text{ is strictly negative definite on }\ker(\omega\cdot\Id-\gamma(t))\text{ for all }t\in\mathbb{R}\text{ and }\omega\in U,
\end{equation}
  then, all the results still hold under a time reversal $t\mapsto A(t)$. But if we remove ``positive'' from \eqref{defn: weaker definiteness assumption}, i.e., if we assume \begin{equation}\label{defn: weaker definiteness assumption 3}
 A(t)\text{ is strictly definite on }\ker(\omega\cdot\Id-\gamma(t))\text{ for all }t\in\mathbb{R}\text{ and }\omega\in U,
\end{equation}
then the system is a mixture of positive and negative systems, which is locally decomposable. To be more precise, we denote by $\Lambda_{+}(t)$ (resp. $\Lambda_{-}(t)$) the eigenvalues $\omega$ on the unit circle $U$ such that $A(t)$ is strictly positive (resp. negative) definite on $\ker(\omega\cdot\Id-\gamma(t))$. Under the condition \eqref{defn: weaker definiteness assumption 3}, the Hausdorff distance between the two sets $\Lambda_{+}(t)$ and $\Lambda_{-}(t)$ is strictly positive and lower semi-continuous in $t$. By Lemma~\ref{lem: dim reduction}, locally as $t$ varies, the eigenvalues are separated into two groups. The first group corresponds to a possibly smaller system satisfying \eqref{defn: weaker definiteness assumption} and the second group corresponds to a system satisfying \eqref{defn: weaker definiteness assumption 2}.
\end{remark}

The proof of Theorem~\ref{thm: krein lyubarskii c1} a) is different from previous argument by Krein, Lyubarskii and Ekeland. Besides, our argument is direct and elementary. We analyze the asymptotics of coefficients of the characteristic polynomial of $\gamma(t)$. This is linked to the Jordan structure of the symplectic matrix via exterior products of linear maps. By continuity of roots depending on the coefficients of certain properly normalized polynomial, we deduce the asymptotics of eigenvalues. This part is some sort of blowup analysis. For the part b) of Theorem~\ref{thm: krein lyubarskii c1}, we use Theorem~\ref{thm: krein lyubarskii c1} a) together with a local $C^{1}$-approximation of $t\mapsto \gamma(t)$ by analytic symplectic paths. Indeed, Theorem~\ref{thm: krein lyubarskii c1} a) provides an upper bound for the number of Krein definite eigenvalues on the circle by first order asymptotics of the eigenvalues. On the other hand, the approximation argument provides matching lower bounds. However, such an approximation argument alone is not sufficient to predict the movement of eigenvalues. We have to combine it with the monotonicity of certain index function, see Claim~\ref{claim: non-decrease nu M t}. As an intermediate step, in the appendix, we sketch the argument of Theorem~\ref{thm: krein lyubarskii c1} when $t\mapsto A(t)$ is real analytic.

\subsection{Organization of the paper}
We collect definitions and notations, prepare some useful properties in Section~\ref{sect: preliminaries}. We prove Theorem~\ref{thm: krein lyubarskii c1} a) in Section~\ref{sect: proof of KL C1 a} and Theorem~\ref{thm: krein lyubarskii c1} b) in Section~\ref{sect: proof of KL C1 b}. We sketch the argument of Theorem~\ref{thm: krein lyubarskii c1} for the analytic case in Subsection~\ref{subsect: krein lyubarskii analytic}.

\section{Preliminaries}\label{sect: preliminaries}
\subsection{Notations and definitions}\label{subsect: notation}
\begin{itemize}
 \item For two positive integers $m$ and $n$, we denote by $M_{m\times n}(\mathbb{C})$ (resp. $M_{m\times n}(\mathbb{R})$) the set of $m\times n$ complex (resp. real) matrices. When $m=n$, we use the notations $M_{n}(\mathbb{C})$ and $M_{n}(\mathbb{R})$ for simplicity. For a square matrix, we define its size as the number of rows in the matrix.
 \item For a matrix $M$, we denote by $M^{T}$ the transpose of $M$. For a complex matrix $M$, we denote by $M^{*}$ the conjugate transpose of $M$.
 \item For $n\geq 1$, we denote by $\Id_n$ the $n\times n$ identity matrix and define $J_{2n}\overset{\text{def}}{=}\begin{bmatrix}0 & \Id_n\\-\Id_n & 0\end{bmatrix}$. Then, $J_{2n}^{*}=J_{2n}^{T}=-J_{2n}$ and $J_{2n}^{2}=-\Id$.
 \item For a vector space $V$ and a finite number of subspaces $\{V_i\}_{i\in I}$, we denote by $\sum_{i\in I}V_i$ the sum of the vector spaces $\sum_{i\in I}V_i$.
 \item For vectors $v_1,\ldots,v_n$ in a vector space $V$, we denote by $\wedge_{j=1}^{n}v_j$ the exterior product $v_1\wedge v_2\wedge\cdots\wedge v_n$. (Note that $\wedge$ is associative.) We denote by $\Lambda^{n}(V)$ the linear span of all such $\wedge_{j=1}^{n}v_j$ and denote by $\Lambda(V)$ the direct sum $\bigoplus_{n\geq 0}\Lambda^{n}(V)$ with the convention that $\Lambda^{0}(V)=\{0\}$. For a totally ordered set $P=\{p_1,\ldots,p_n\}$ with $p_1\prec p_2\prec\cdots\prec p_n$ and vectors $(v_p)_{p\in P}$ indexed by $P$, we denote by $\wedge_{p\in P}v_p$ the exterior product $v_{p_1}\wedge v_{p_2}\wedge\cdots\wedge v_{p_n}$. (Note that $\Lambda(V)$ is a vector space. Hence, if we take $(v_p)_{p\in P}$ from the vector space $\Lambda(V)$, then we define the exterior products of exterior products in a consistent manner.)
 \item For $m\geq 1$, the inner product $\langle\cdot,\cdot\rangle$ on $\mathbb{C}^{m}$ is defined by \[\langle x,y\rangle=\sum_{j=1}^{m}x_j\bar{y}_{j}.\] Then, for $x,y\in\mathbb{C}^{2n}$,
     \[(x,y)_{G}=\sqrt{-1}\langle x,J_{2n}y\rangle=-\sqrt{-1}\langle J_{2n}x,y \rangle=\sqrt{-1}\left\{\sum_{k=1}^{n}(x_n\bar{y}_{n+k}-x_{n+k}\bar{y}_{k})\right\}.\]
 \item For $n\geq 1$ and a linear subspace $V$ of $\mathbb{C}^{2n}$, we denote by $V^{\perp_G}$ the \emph{symplectic orthogonal complement} of $V$, i.e.,
 \[V^{\perp_{G}}=\{x\in\mathbb{C}^{2n}:(x,y)_{G}=0, \forall y\in V\}.\]
 The linear subspace $V$ is \emph{symplectic} if $V\cap V^{\perp_{G}}=\{0\}$. When $V$ is a linear subspace of $\mathbb{R}^{2n}$, we replace $\mathbb{C}^{2n}$ by $\mathbb{R}^{2n}$ in the above definition.
 \item For a $k\times k$ complex valued matrix $M$ and an eigenvalue $\lambda$ of $M$, the geometric multiplicity of $\lambda$ is defined as $\dim\ker(\lambda\cdot\Id-M)$ and the algebraic multiplicity is defined as $\dim\ker(\lambda\cdot\Id-M)^{k}$. We denote by $E_{\lambda}=E_{\lambda}(M)$ the invariant subspace of $\mathbb{C}^{k}$, i.e., \[E_{\lambda}=\{x\in\mathbb{C}^k:(\lambda\cdot\Id-M)^{k}x=0\}.\]
 \item Denote by $p(\lambda,t)$ the characteristic polynomial of the matrix $\gamma(t)$, i.e.,
 \[p(\lambda,t)=\det(\lambda\cdot\Id-\gamma(t)).\]
\end{itemize}

\subsection{Exterior powers of linear maps}\label{subsect: exterior powers of linear maps}
We recall exterior powers of a linear map $A$ and its relation with its determinant $\det(A)$.

Starting from several linear maps on a vector space $V$, there are many ways to combine them to define multi-linear skew symmetric maps (or equivalently, linear maps on the exterior products $\Lambda^{m}(V)$ of $V$). We follow the construction in \cite[Section~3.7]{Winitzki2010}. For natural numbers $k\leq m$, the author defines a linear map on $\Lambda^{m}(V)$ by taking certain ``skew symmetrization'' of tensors of $k$ many linear maps $A$ with $m-k$ many identity maps. For our purpose, it suffices to take $m$ to be the dimension of $V$. But we need a slightly generalization to allow the combination of three linear maps $A_1$, $A_2$ and the identity map. We introduce these notations in the following definition.

\begin{defn}\label{defn: ext powers of linear maps}
 Let $A:V\to V$ be a linear map on an $n$-dimensional vector space $V$. For $k=0,\ldots,n$, we define the exterior powers of $\bigwedge(n,k,A):\Lambda^{n}(V)\to \Lambda^{n}(V)$ as a linear map as follows:
 \begin{equation}
  \bigwedge(n,k,A)(v_1\wedge\cdots \wedge v_n)\overset{\mathrm{def}}{=}\sum_{\sigma\in\{0,1\}^{n}:\sum_{i}\sigma_i=k}\wedge^{n}_{i=1}(\sigma_i\cdot Av_i+(1-\sigma_i)\cdot v_i).
 \end{equation}
 Similarly, for linear maps $A_1,A_2:V\to V$, for $k_1,k_2=0,\ldots,n$, we define the linear map $\bigwedge(n,k_1,k_2,A_1,A_2):\Lambda^{n}(V)\to \Lambda^{n}(V)$ as follows:
\begin{multline}
 \bigwedge(n,k_1,k_2,A_1,A_2)(v_1\wedge\cdots \wedge v_n)\\
 \overset{\mathrm{def}}{=}\sum_{\sigma\in\{0,1,2\}^{n}:\sum_{i}1_{\sigma_i=1}=k_1,\sum_{i}1_{\sigma_i=2}=k_2}\wedge^{n}_{i=1}(1_{\sigma_i=1}\cdot A_1v_i+1_{\sigma_i=2}\cdot A_2v_i+1_{\sigma_i=0}\cdot v_i),
\end{multline}
 Since $\Lambda^{n}(V)$ is $1$-dimensional, we identify the $\bigwedge(n,k,A)$ (or $\bigwedge(n,k_1,k_2,A_1,A_2)$) with the unique scaling factor, which is also denoted by $\bigwedge(n,k,A)$ (or $\bigwedge(n,k_1,k_2,A_1,A_2)$).
\end{defn}

In the above definition, for each vector $v_i$, we choose one from the three linear maps $\Id$, $A_1$ and $A_2$ and apply it to $v_i$. For the assignment of linear maps to the linear basis, the only constraint is that the map $A_1$ occurs $k_1$ many times and the map $A_2$ occurs exactly $k_2$ many times. All these assignments have equal weight.

Note that $\det(A)$ is identified with the linear map $\bigwedge(n,n,A)$ on the $1$-dimensional vector space $\Lambda^{n}(V)$. In particular, for an eigenvalue $\lambda_0$ of the matrix $\gamma(0)$, we have that
\begin{multline}\label{eq: char polynomial three terms}
 p(\lambda,t)=\det(\lambda\cdot\Id-\gamma(t))=\det((\lambda-\lambda_0)\cdot\Id+(\lambda_0\cdot\Id-\gamma(0))-(\gamma(t)-\gamma(0)))\\
 =\sum_{k=0}^{2n}(\lambda-\lambda_0)^{k}\sum_{k_1+k_2=2n-k,k_1\geq 0,k_2\geq 0}(-1)^{k_2}\cdot\bigwedge(2n,k_1,k_2,\lambda_0\cdot\Id-\gamma(0),\gamma(t)-\gamma(0)).
\end{multline}
In the above calculation, we express the determinant by wedge powers of the sum of linear maps $(\lambda-\lambda_0)\cdot\Id$, $\lambda_0\cdot\Id-\gamma(0)$ and $\gamma(0)-\gamma(t)$, expand it according to distributive law and collect the terms with the same times of occurrence, where $k_1$ is the time of occurrence of $\lambda_0\cdot\Id-\gamma(0)$ and $k_2$ counts the occurrence of $\gamma(0)-\gamma(t)$.

\subsection{Continuity of roots of polynomials}
Consider a polynomial with complex coefficients of degree at most $n$. We will need the following lemma on the continuity of the roots as the coefficients vary.
\begin{lemma}\label{lem: continuity of roots of general polynomials}
 Let $W$ be a neighborhood of $0$. Let $P_{t}(z)=\sum_{j=0}^{n}c_{j}(t)z^{j}$, where $c_{j}(t)\in\mathbb{C}$ and $t\in W$. Suppose that $t\mapsto c_j(t)$ is continuous for $j=0,\ldots,n$ and $t\in W$. Denote by $d(t)$ the degree of the polynomial $P_{t}$. Suppose that $d(t)=n$ for $t\in W\setminus\{0\}$ and $d(0)=m\leq n$. Then, there exist $m$ continuous complex valued functions $z_1,\ldots,z_{m}$ on $W$ and $n-m$ continuous complex valued functions $z_{m+1},\ldots,z_{n}$ on $W\setminus\{0\}$ such that
 \begin{itemize}
  \item for $t\neq 0$, $z_1(t),\ldots,z_n(t)$ are roots of $P_{t}$,
  \item for $t=0$, $z_1(0),\ldots,z_m(0)$ are roots of $P_0$,
  \item for $i=m+1,\ldots,n$, we have that $\lim_{t\to 0}z_i(t)=\infty$.
 \end{itemize}
\end{lemma}
\begin{proof}
 By assumptions, for $t_0\in W$, $P_{t}(z)\overset{t\to t_0}{\to}P_{t_0}(z)$ uniformly for $z$ on compacts. Hence, for any continuous loop $\Gamma$ avoiding the roots of $P_{t_0}$, for $t$ sufficiently close to $t_0$, $P_{t}$ does not vanish on $\Gamma$ and
 \begin{equation}\label{eq: contour integral converge as epsilon to delta}
  \lim_{t\to t_0}\frac{1}{2\pi\sqrt{-1}}\int_{\Gamma}\frac{1}{P_{t}(z)}\,\mathrm{d}z=\frac{1}{2\pi\sqrt{-1}}\int_{\Gamma}\frac{1}{P_{t_0}(z)}\,\mathrm{d}z.
 \end{equation}
 Also, note that for a simple loop avoiding the roots of $P_{t}$, $\frac{1}{2\pi\sqrt{-1}}\int_{\Gamma}\frac{1}{P_{t}(z)}\,\mathrm{d}z$ is precisely the number of roots inside the loop. (The interior and exterior region are determined by the orientation of the loop.) Eventually, Lemma~\ref{lem: continuity of roots of general polynomials} holds since \eqref{eq: contour integral converge as epsilon to delta} holds for all continuous loops avoiding the roots of $P_{t_0}$.

\end{proof}

\subsection{Properties of symplectic matrices}
We collect some well-known properties of symplectic matrices in this subsection. The following observations, although elementary, are frequently used in some calculation. For a complex number $\lambda$, the adjoint of $\lambda\cdot\Id$ under $(\cdot,\cdot)_{G}$ is $\bar{\lambda}\cdot\Id$, i.e., $\forall x,y\in\mathbb{C}^{2n}$, we have
 \begin{equation}
  (\lambda x,y)_{G}=(x,\bar{\lambda} y)_{G}.
 \end{equation}
 For a symplectic matrix $\gamma$, the adjoint of $\gamma$ under $(\cdot,\cdot)_{G}$ is $\gamma^{-1}$, i.e., $\forall x,y\in\mathbb{C}^{2n}$, we have
 \begin{equation}
  (\gamma x,y)_{G}=(x,\gamma^{-1} y)_{G}.
 \end{equation}
 For all symplectic subspaces $V$, the restriction of the bilinear form $(\cdot,\cdot)_{G}$ on $V$ is non-degenerate. For a symplectic subspace $V$, if it is invariant under the linear symplectic transform $\gamma$, then so is its symplectic orthogonal complement $V^{\perp_{G}}$.

The following criteria on the $G$-orthogonality of invariant spaces is basically \cite[Proposition 5, Section 2, Chapter 1]{EkelandMR1051888}.
\begin{lemma}\label{lem: G othorgonal invariant spaces}
 Let $\lambda$ and $\mu$ be two eigenvalues of the symplectic matrix $\gamma\in\Sp(2n,\mathbb{R})$. If $\lambda\bar{\mu}\neq 1$, then the invariant spaces $E_{\lambda}$ and $E_{\mu}$ are $G$-othorgonal. Consider a partition $\{P_1,\ldots,P_k\}$ of the set of eigenvalues of $\gamma$ such that each $P_i$ is stable under the circular reflection $z\mapsto \bar{z}^{-1}$. For each $i$, let $E_{i}=\cup_{\lambda\in P_i}E_{\lambda}$. Then, $E_1,\ldots,E_k$ is a $G$-orthogonal decomposition of $\mathbb{C}^{2n}$. In particular, when $\lambda\in U$, we have the following $G$-othorgonal decomposition of $\mathbb{C}^{2n}$:
 \begin{equation}
  \mathbb{C}^{2n}=E_{\lambda}\oplus F_{\lambda},
 \end{equation}
 where $F_{\lambda}$ is the direct sum of $\{E_{\mu}\}_{\mu\neq\lambda}$. Hence, if $\lambda$ is a simple eigenvalue on $U$, it is Krein definite.
\end{lemma}

The ``inner product'' under $(\cdot,\cdot)_{G}$ of the generalized eigenvectors in \eqref{eq: gammaxi=lambda0xi-xi} and \eqref{eq: xiij basis} must satisfy certain algebraic relations:
\begin{lemma}\label{lem: vector G products in same invariant space}
 Suppose that $\lambda$ is an eigenvalue of the symplectic matrix $\gamma$. We use the same notations $\{\xi_{i,j}\}_{i=1,\ldots,m;j=1,\ldots,j_i}$ and $\{\eta_{i,j}\}_{i=1,\ldots,m;j=1,\ldots,j_i}$ as \eqref{eq: gammaxi=lambda0xi-xi}, \eqref{eq: xiij basis}, \eqref{eq: xiij regroup} and \eqref{eq: definition of etaij} for the eigenvalue $\lambda$ of $\gamma$ instead of the eigenvalue $e^{\sqrt{-1}\theta_0}$ of $\gamma(0)$. For $i,i'=1,\ldots,m$, $j=0,\ldots,j_i-1$ and $j'=0,\ldots,j_{i'}-1$, we have that
 \begin{equation}\label{eq: xiij G product binom like}
  (\lambda^{j}\xi_{i,j},\lambda^{j'}\xi_{i',j'})_{G}=(\lambda^{j+1}\xi_{i,j+1},\lambda^{j'}\xi_{i',j'})_{G}+(\lambda^{j}\xi_{i,j},\lambda^{j'+1}\xi_{i',j'+1})_{G},
 \end{equation}
 with the convention that $\xi_{i,0}=0$ for all $i=1,\ldots,m$. In particular, when $j+j'\leq \max(j_i,j_{i'})$, we have that
 \[(\eta_{i,j},\eta_{i',j'})_{G}=(\xi_{i,j},\xi_{i',j'})_{G}=0.\]
 For fixed $i,i'=1,\ldots,s$, we have the same value $(\eta_{i,j},\eta_{i',j'})_{G}$ for all $j=1,\ldots,j_i$ and $j'=1,\ldots,j_{i'}$ such that $j+j'=\max(j_i,j_{i'})+1$.
\end{lemma}
\begin{proof}
 Since $\gamma$ is sympletic, we have that $(\gamma v_1,\gamma v_2)_{G}=(v_1,v_2)_{G}$ for all $v_1,v_2\in\mathbb{C}^{2n}$. By taking $v_1=\xi_{i,j+1}$ and $v_2=\xi_{i',j'+1}$, we obtain \eqref{eq: xiij G product binom like}. The rest directly follows from \eqref{eq: xiij G product binom like}.
\end{proof}

Note that $(x,y)_{G}=\overline{(y,x)_{G}}$. From non-degeneracy of $(\cdot,\cdot)_{G}$ and Lemmas~\ref{lem: G othorgonal invariant spaces} and \ref{lem: vector G products in same invariant space}, we get
\begin{corollary}\label{cor: Xellell hermitian non-degenerate xell1ell2 vanishes}
 Recall the notations \eqref{eq: defn X xiJxi} and \eqref{eq: SX blocks}. For $\ell=1,\ldots,s$, the matrix $X^{(\ell,\ell)}$ is Hermitian and non-degenerate. For $1\leq \ell_1<\ell_2\leq s$, we have that $X^{(\ell_2,\ell_1)}=0$.
\end{corollary}

\section{Proof of Theorem~\ref{thm: krein lyubarskii c1} a)}\label{sect: proof of KL C1 a}
As the proof of Theorem~\ref{thm: krein lyubarskii c1} a) is long and technical, we decide to give the sketch of the proof and provide some intuitive ideas in advance. Suppose $\lambda_0=e^{\sqrt{-1}\theta_0}\in U$ is an eigenvalue of $\gamma(0)$. We expand the characteristic polynomial $p(\lambda,t)=\det(\lambda\cdot\Id-\gamma(t))$ at $e^{\sqrt{-1}\theta_0}$:
\begin{equation}\label{defn: ck expansion at lambda0}
 p(\lambda,t)=\sum_{k=0}^{2n}c_{k}(t)(\lambda-e^{\sqrt{-1}\theta_0})^k.
\end{equation}
In order to study the asymptotics of the eigenvalues as $t$ varies from $0$, we study the asymptotics of the coefficients $\{c_k(t)\}_{k=0,\ldots,2n}$ in Lemma~\ref{lem: asymptotics coeff char polynomial} in Subsection~\ref{subsect: proof of lem asymp coeff char poly}. We will illustrate the results of Lemma~\ref{lem: asymptotics coeff char polynomial} and explain the way to prove Theorem~\ref{thm: krein lyubarskii c1} a) from Lemma~\ref{lem: asymptotics coeff char polynomial} by a concrete example. But we will not sketch the technical proof of Lemma~\ref{lem: asymptotics coeff char polynomial}.

In Lemma~\ref{lem: asymptotics coeff char polynomial}, we will show that $c_k(t)=O(t^{\varphi(k)})$ as $t\mapsto 0$, where $\varphi(k)$ is a certain integer valued function in $k$. Let us precisely give the value of $\varphi(k)$. Denote by $N$ the algebraic multiplicity of $\lambda_0$. Then, $\varphi(k)$ is simply $0$ for $k\geq N$. For $k=0,\ldots,N-1$, the value of $\varphi(k)$ can be obtained graphically via Young diagrams as follows: we list the sizes of Jordan blocks associated with $\lambda_0$ in non-increasing order $j_1\geq j_2\geq \cdots\geq j_m$. The sequence $\{j_i\}_{i=1,\ldots,m}$ forms a partition of $N$ and is represented by a Young diagram. The Young diagram consists of unit squares placed side by side. For $i=1,\ldots,m$, the $i$-th row has exactly $j_{m+1-i}$ many squares. All these rows are aligned to the left. Please see Figure~\ref{fig: youngdiagram} for the Young diagram associated with the partition $4+2+2+1=9$.
\begin{figure}
\centering
\begin{minipage}[c]{0.4\textwidth}
\label{fig: youngdiagram}
\includegraphics{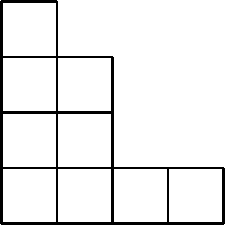}
\caption{Young diagram}
\end{minipage}
\begin{minipage}[c]{0.4\textwidth}
\label{fig: youngdiagramphi1}
\includegraphics{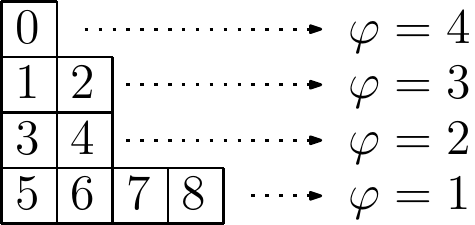}
\caption{$k\mapsto\varphi(k)$}
\end{minipage}
\end{figure}
To get the value of $\varphi$, we fill the diagram with integers $\{0,\ldots,N-1\}$ from the top row to the bottom row. In each row, we fill the diagram from the left to the right. Then, each integer $k$ is filled in the $\varphi(k)$-th row from the bottom, see Figure~\ref{fig: youngdiagramphi1}. Alternatively, from a finite non-increasing sequence of integers $\{j_i\}_{i=1,\ldots,m}$, their partial sums $\{\sum_{i=k}^{m}j_i\}_{k=1,\ldots,m}$ form a strictly decreasing sequence, the upper boundary of the corresponding new Young diagram represents the graph of the function $\varphi(k)$, see Figure~\ref{fig: youngdiagramphi2} for the same sizes of Jordan blocks as Figures~\ref{fig: youngdiagram} and \ref{fig: youngdiagramphi1}.
\begin{figure}
\centering
\label{fig: youngdiagramphi2}
\includegraphics{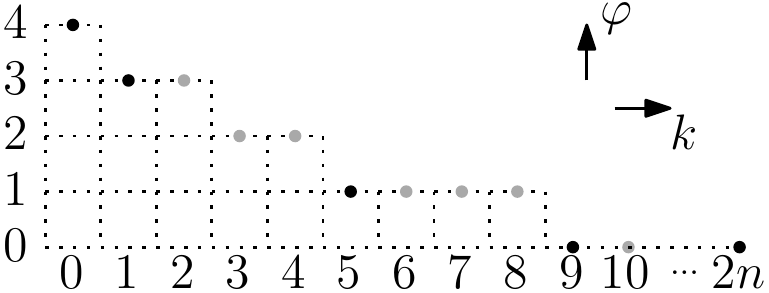}
\caption{$k\to\varphi(k)$}
\end{figure}
The black and grey points give the graph of $\varphi$. (Recall that $\varphi$ is set to $0$ for $k\geq N$ and $N=9$ in the above figures.)

Let us explain the difference between black and grey points in the following. Roughly speaking, the black points separate the Jordan blocks with different sizes. Alternatively, the black points are exactly the extremal points of the convex hull of the discrete domain $\{(\tilde{k},\tilde{\varphi}):\tilde{k}=0,\ldots,2n,\tilde{\varphi}\in\mathbb{Z},\tilde{\varphi}\geq \varphi(\tilde{k})\}$ above the graph of $\varphi$, see Figure~\ref{fig: youngdiagramphiaboveconvexhull}.
\begin{figure}
\centering
\label{fig: youngdiagramphiaboveconvexhull}
\includegraphics{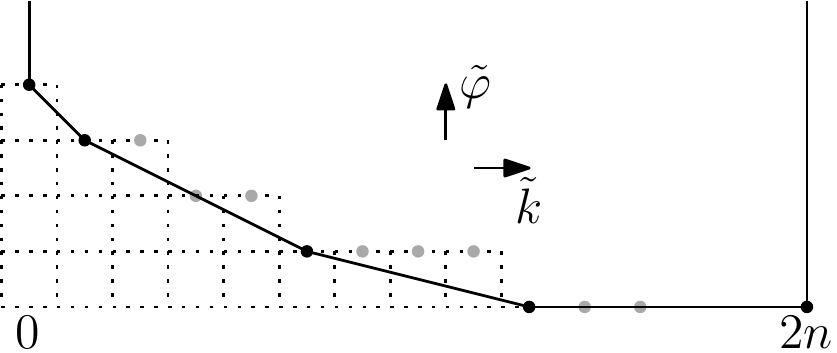}
\caption{Boundary of the convex hull of $\{(\tilde{k},\tilde{\varphi}):\tilde{k}=0,\ldots,2n,\tilde{\varphi}\in\mathbb{Z},\tilde{\varphi}\geq \varphi(\tilde{k})\}$}
\end{figure}
We will prove in Lemma~\ref{lem: asymptotics coeff char polynomial} that $c_{k}(t)=O(t^{\varphi(k)})$ as $t\to 0$. For general $k$ (corresponding to the grey dots), $\varphi(k)$ is not necessarily the exact order of $c_k(t)$. However, for those $k$ corresponding to the black dots, the order $\varphi(k)$ is exact and we calculate $\lim_{t\to 0}c_k(t)t^{-\varphi(k)}$ in \eqref{eq: ck asymp entire block} of Lemma~\ref{lem: asymptotics coeff char polynomial}.

Next, we sketch the proof of Theorem~\ref{thm: krein lyubarskii c1} a) from Lemma~\ref{lem: asymptotics coeff char polynomial}. We will carry out certain blow up analysis at $t=0$ and $\lambda=\lambda_0$. Take $w=\frac{\lambda-\lambda_0}{t^{\alpha}}$ for $\alpha>0$\footnote{For $\alpha=0$, the following argument simply yields the continuity of the eigenvalues of $\gamma(t)$ as $t$ varies from $0$.}. After a change of variable, we obtain another polynomial $q_{\alpha}(w,t)$ from $p(\lambda,t)$, where $q_{\alpha}(w,t)=\sum_{k=0}^{2n}c_{k}(t)t^{\alpha k}w^{k}$. Note that $\lim_{t\to 0}q_{\alpha}(w,t)=0$. To obtain a non-trivial limit, we need to divide $q_{\alpha}(w,t)$ by $t^{\beta(\alpha)}$, where $\beta(\alpha)=\min\{\varphi(k)+\alpha k:k=0,\ldots,2n\}$. We are interested in the limiting polynomial
\[r_{\alpha}(w)=\sum_{k=0}^{2n}\lim_{t\to 0}c_k(t)t^{\alpha k-\beta(\alpha)}\cdot w^{k}.\]
In order to obtain $\lim_{t\to 0}\frac{\lambda(t)-\lambda_0}{t^{\alpha}}$ by using Lemma~\ref{lem: continuity of roots of general polynomials}, we need to answer the following questions: does $r_{\alpha}(w)$ vanish? If not, how to describe the roots of $r_{\alpha}(w)$?

Note that the possible minimizers of $\varphi(k)+\alpha k$ are important to us since
\[\lim_{t\to 0}c_{k}(t)t^{\alpha k-\beta(\alpha)}=\left\{\begin{array}{ll}
0 & \text{ if }k\text{ is not a minimizer,}\\
\lim\limits_{t\to 0}c_{k}(t)t^{-\varphi(k)} & \text{ if }k\text{ is a minimizer}.
\end{array}\right.\]
Denote by $L_{\alpha}$ the line through the origin with the slope $-\alpha$. To find the minimizers, we translate $L_{\alpha}$ upwards until $L_{\alpha}$ has non-empty intersection with the graph of $\varphi(k)$ for the first time. The $k$-coordinates of the intersection points are precisely the minimizers. The intersection must contain black points since the black points are extremal points of the convex hull of the discrete domain above the graph of $\varphi$, see Figure~\ref{fig: youngdiagramphiaboveconvexhull}. For the $k$-coordinates of the black intersection points, the limit $\lim_{t\to 0}c_{k}(t)t^{\alpha k-\beta(\alpha)}=\lim_{t\to 0}c_{k}(t)t^{-\varphi(k)}\neq 0$. In particular, $r_{\alpha}(w)\not\equiv 0$.

When $\frac{1}{\alpha}$ is different from the sizes of Jordan blocks associated with $\lambda_0$, the minimizer $k_{\min}$ is the single black intersection point, see Figure~\ref{fig: youngdiagramphiplusgenericalphak} for $\alpha=\frac{3}{4}$ and the same sizes of Jordan blocks as in Figure~\ref{fig: youngdiagram}.
\begin{figure}
\begin{minipage}[b]{0.45\textwidth}
\label{fig: youngdiagramphiplusgenericalphak}
\includegraphics[width=0.95\textwidth]{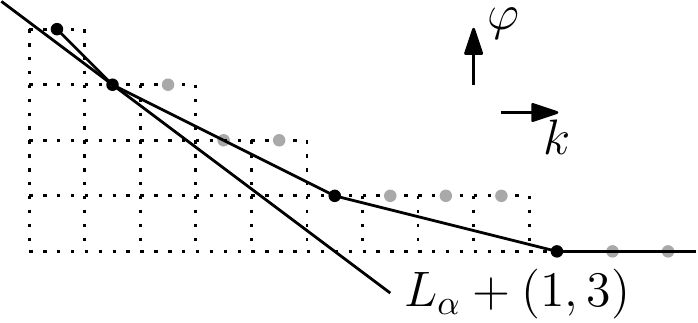}
\caption{Generic intersection}
\end{minipage}
\begin{minipage}[b]{0.45\textwidth}
\label{fig: youngdiagramphiplushalfk}
\includegraphics[width=0.95\textwidth]{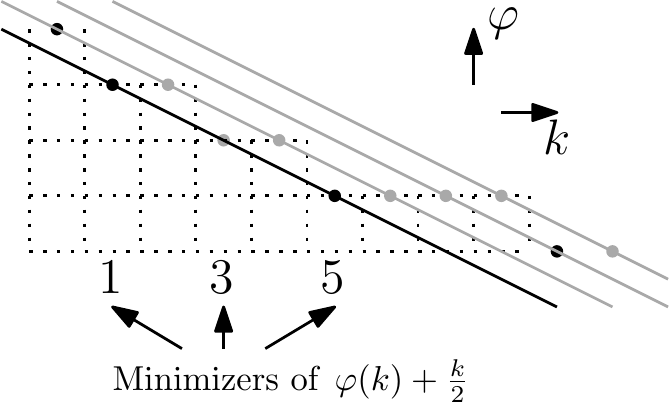}
\caption{$\alpha=\frac{1}{2}$}
\end{minipage}
\end{figure}
In this case, the limiting polynomial $r_{\alpha}(w)$ consists of a single term and its roots must be zero. When $\frac{1}{\alpha}$ equals the size of a Jordan block associated with $\lambda_0$, there are exactly $m(\alpha)+1$ many minimizers, where $m(\alpha)$ equals the number of Jordan blocks (associated with $\lambda_0$) of the size $\frac{1}{\alpha}$, see Figure~\ref{fig: youngdiagramphiplushalfk}. The minimizers have equal distance $\frac{1}{\alpha}$ between each other (since the intermediate grey points separate Jordan blocks of the same size). By Lemma~\ref{lem: asymptotics coeff char polynomial}, the coefficients of the limiting polynomial $r$ correspond to the sum of certain principle minors. Finally, we write $r$ as certain determinant and the asymptotics of eigenvalues are determined by the sizes of Jordan blocks and the roots of the limiting polynomial $r$. For instance, in Figure~\ref{fig: youngdiagramphiplushalfk}, the sizes of Jordan blocks are $4$, $2$, $2$ and $1$, $\alpha=\frac{1}{2}$ corresponds to the Jordan blocks of size $2$ and $N=9$. The minimizers are $1$, $3$ and $5$ and $\varphi(1)=3$, $\varphi(3)=2$ and $\varphi(5)=1$. By Lemma~\ref{lem: asymptotics coeff char polynomial}, the limiting polynomial
\begin{align*}
r_{\frac{1}{2}}(w)=&\sum_{k=1,3,5}\lim_{t\to 0}c_{k}(t)t^{-\varphi(k)}\cdot w^{k}\\
=&c_{9}(0)w\left(\begin{vmatrix}
d_{1,1} & d_{1,2} & d_{1,3}\\
d_{2,1} & d_{2,2} & d_{2,3}\\
d_{3,1} & d_{3,2} & d_{3,3}
\end{vmatrix}+w^2\left(\begin{vmatrix}
d_{1,1} & d_{1,2}\\
d_{2,1} & d_{2,2}
\end{vmatrix}+\begin{vmatrix}
d_{1,1} & d_{1,3}\\
d_{3,1} & d_{3,3}
\end{vmatrix}\right)+w^{4}d_{1,1}\right)\\
=&c_9(0)w\begin{vmatrix}
d_{1,1} & d_{1,2} & d_{1,3}\\
d_{2,1} & d_{2,2}+w^2 & d_{2,3}\\
d_{3,1} & d_{3,2} & d_{3,3}+w^2
\end{vmatrix},
\end{align*}
where the $4\times 4$ matrix $d$ equals $SX^{-1}\Lambda$, where $\Lambda$ is a diagonal matrix with diagonal elements $-\lambda_0^4$, $\lambda_0^2$, $\lambda_0^2$ and $\sqrt{-1}\lambda_0$ from the top to the bottom. (Recall the definition of $S$ and $X$ in \eqref{eq: defn S Cxixi} and \eqref{eq: defn X xiJxi}.) A non-trivial root of $r_{\frac{1}{2}}(w)$ corresponds to a non-zero finite limit of $\lim_{t\to 0}\frac{\lambda(t)-\lambda_0}{t^{\frac{1}{2}}}$ where $\lambda(t)$ is an eigenvalue of $\gamma(t)$. It is H\"{o}lder-$\frac{1}{2}$ continuous at $t=0$. The trivial root $0$ of $r_{\frac{1}{2}}(w)$ corresponds to the zero limit of $\lim_{t\to 0}\frac{\lambda(t)-\lambda_0}{t^{\frac{1}{2}}}$, where $\lambda(t)$ corresponds to certain Jordan block of strictly smaller size and has better regularity at $t=0$. The non-trivial roots of $r_{\alpha}(w)$ are important and they are also the roots of the polynomial $Q(w)$, where
\begin{equation}\label{eq: Q(w) special case}
Q(w)=\begin{vmatrix}
d_{1,1} & d_{1,2} & d_{1,3}\\
d_{2,1} & d_{2,2}+w^2 & d_{2,3}\\
d_{3,1} & d_{3,2} & d_{3,3}+w^2
\end{vmatrix}.
\end{equation}
Write the matrices $d$, $S$, $X$ and $\Lambda$ as in \eqref{eq: SX blocks} with $s=3$:
\begin{align*}
&d=\left(\begin{array}{c|cc|c}
d_{1,1} & d_{1,2} & d_{1,3} & d_{1,4}\\
\hline
d_{2,1} & d_{2,2} & d_{2,3} & d_{2,4}\\
d_{3,1} & d_{3,2} & d_{3,3} & d_{3,4}\\
\hline
d_{4,1} & d_{4,2} & d_{4,3} & d_{4,4}
\end{array}\right),\quad &\Lambda=\left(\begin{array}{c|cc|c}
-\lambda_0^4 & 0 & 0 & 0\\
\hline
0 & \lambda_0^2 & 0 & 0\\
0 & 0 & \lambda_0^2 & 0\\
\hline
0 & 0 & 0 & \sqrt{-1}\lambda_0
\end{array}\right),\\
&S=\left(\begin{array}{c|cc|c}
S_{1,1} & S_{1,2} & S_{1,3} & S_{1,4}\\
\hline
S_{2,1} & S_{2,2} & S_{2,3} & S_{2,4}\\
S_{3,1} & S_{3,2} & S_{3,3} & S_{3,4}\\
\hline
S_{4,1} & S_{4,2} & S_{4,3} & S_{4,4}
\end{array}\right),\quad &X=\left(\begin{array}{c|cc|c}
X_{1,1} & X_{1,2} & X_{1,3} & X_{1,4}\\
\hline
0 & X_{2,2} & X_{2,3} & X_{2,4}\\
0 & X_{3,2} & X_{3,3} & X_{3,4}\\
\hline
0 & 0 & 0 & X_{4,4}
\end{array}\right).
\end{align*}
By calculation in blocks, we get that
\[
\left(\begin{array}{c|cc}
d_{1,1} & d_{1,2} & d_{1,3}\\
\hline
d_{2,1} & d_{2,2} & d_{2,3}\\
d_{3,1} & d_{3,2} & d_{3,3}
\end{array}\right)=
\left(\begin{array}{c|cc}
S_{1,1} & S_{1,2} & S_{1,3}\\
\hline
S_{2,1} & S_{2,2} & S_{2,3}\\
S_{3,1} & S_{3,2} & S_{3,3}
\end{array}\right)\left(\begin{array}{c|cc}
X_{1,1} & X_{1,2} & X_{1,3}\\
\hline
0 & X_{2,2} & X_{2,3}\\
0 & X_{3,2} & X_{3,3}
\end{array}\right)^{-1}\left(\begin{array}{c|cc}
-\lambda_0^4 & 0 & 0\\
\hline
0 & \lambda_0^2 & 0\\
0 & 0 & \lambda_0^2
\end{array}\right).\]
Write the above equation by $\tilde{d}=\tilde{S}\tilde{X}^{-1}\tilde{\Lambda}$. Denote by $\tilde{I}$ the $3\times 3$ square matrix $\begin{pmatrix}
0 & 0\\
0 & \Id_2\end{pmatrix}$. Then, we have that
\begin{equation*}
Q(w)=\det(\tilde{d}+w^2\tilde{I})=\det(\tilde{S}\tilde{X}^{-1}\tilde{\Lambda}+w^2\tilde{I})=\frac{\det\tilde{\Lambda}}{\det\tilde{X}}\cdot\det(S+w^2\tilde{I}\tilde{\Lambda}^{-1}\tilde{X}).
\end{equation*}
Hence, the roots of $Q$ coincide with the root of $\tilde{Q}(w)$, where
\begin{equation}\label{eq: tilde Q(w) special case}
\tilde{Q}(w)=\det(S+w^2\tilde{I}\tilde{\Lambda}^{-1}\tilde{X})=\det\left\{\left(\begin{array}{c|cc}
S_{1,1} & S_{1,2} & S_{1,3}\\
\hline
S_{2,1} & S_{2,2} & S_{2,3}\\
S_{3,1} & S_{3,2} & S_{3,3}
\end{array}\right)+w^2\lambda_0^{-2}\left(\begin{array}{c|cc}
0 & 0 & 0\\
\hline
0 & X_{2,2} & X_{2,3}\\
0 & X_{3,2} & X_{3,3}
\end{array}\right)\right\}.
\end{equation}
The above method also works in general case as we shall see in Subsection~\ref{subsect: proof of thm KL a from lem asymp coeff}. In the formal proof, we will replace the geometric arguments by explicit and rigorous analysis.

We state and prove Lemma~\ref{lem: asymptotics coeff char polynomial} in Subsection~\ref{subsect: proof of lem asymp coeff char poly}, where we use the exterior powers of linear maps. We deduce Theorem~\ref{thm: krein lyubarskii c1} a) from Lemma~\ref{lem: asymptotics coeff char polynomial} in Subsection~\ref{subsect: proof of thm KL a from lem asymp coeff}. The reader may firstly skip the technical proof of Lemma~\ref{lem: asymptotics coeff char polynomial} and go directly to the proof of Theorem~\ref{thm: krein lyubarskii c1} a).

\subsection{Proof of Lemma~\ref{lem: asymptotics coeff char polynomial}}\label{subsect: proof of lem asymp coeff char poly}

\begin{lemma}\label{lem: asymptotics coeff char polynomial}
 Consider the solution $\gamma(t)\in \Sp(2n,\mathbb{R})$ of \eqref{eq: ham system R modified initial cont pos} without assuming \eqref{defn: weaker definiteness assumption}. Recall the notations \eqref{eq: gammaxi=lambda0xi-xi}, \eqref{eq: xiij basis}, \eqref{eq: xiij regroup}, \eqref{eq: defn S Cxixi}, \eqref{eq: defn X xiJxi} and \eqref{defn: ck expansion at lambda0}. Denote by $N=N(e^{\sqrt{-1}\theta_0})$ the dimension of the invariant space $E_{e^{\sqrt{-1}\theta_0}}(\gamma(0))$. (Note that $N=\sum_{i=1}^{m}j_i$.)
Then, we have that
\begin{equation}\label{eq: cN0}
c_{N}(0)=\lim_{\lambda\to e^{\sqrt{-1}\theta_0}}\frac{p(\lambda,0)}{(\lambda-e^{\sqrt{-1}\theta_0})^{N}}.
\end{equation}
For $k=0,\ldots,N-1$, as $t\to 0$,
 \begin{equation}\label{eq: expression of ck big O}
  c_k(t)=(-t)^{\varphi(k)}\cdot\bigwedge(2n,2n-k-\varphi(k),\varphi(k),e^{\sqrt{-1}\theta_0}\cdot\Id-\gamma(0),\frac{\mathrm{d\gamma}}{\mathrm{d}t}(0))+o(t^{\varphi(k)}),
 \end{equation}
 where
 \begin{align}\label{defn: varphi(k)}
  \varphi(k)=\varphi(k,\gamma(0))&\overset{\mathrm{def}}{=}\min\left\{i=1,\ldots,m:\sum_{i'=1}^{i}j_{i'}\geq N-k\right\}\notag\\
  &=\min\left\{i=1,\ldots,m:\sum_{i<i'\leq m}j_{i'}\leq k\right\}.
 \end{align}
  (Consider the Jordan blocks associated with the eigenvalue $e^{\sqrt{-1}\theta_0}$. Then, $\varphi(k)$ is precisely the minimal number of blocks such that their total size is not less than $N-k$. By definition, we have that $k\geq \sum_{i>\varphi(k)}j_i$.)
  In particular, when $k=\sum_{i>\varphi(k)}j_{i}$, as $t\to 0$,
  \begin{equation}\label{eq: expression of ck exact N-k}
    c_{k}(t)=(-1)^{N-k}t^{\varphi(k)}c_{N}(0)\sum_{I\in\mathcal{I}_{k}}\det\left[(d_{i,i'})_{i,i'\in I}\right]+o(t^{\varphi(k)}),
  \end{equation}
  where
  \[\mathcal{I}_{k}\overset{\mathrm{def}}{=}\left\{I\subset\{1,\ldots,m\}:\begin{array}{l}
  \sharp I=\varphi(k); \forall i\in I,j_i\geq j_{\varphi(k)}\\
  \text{and }\forall i\in\{1,\ldots,m\}\setminus I,j_i\leq j_{\varphi(k)}
  \end{array}\right\}\quad\footnote{To get the quantity on the right hand side of \eqref{eq: expression of ck exact N-k}, we select the biggest $\varphi(k)$-many Jordan blocks. However, due to possible presence of Jordan blocks of equal size, such a selection is not unique and $\mathcal{I}_k$ is introduced to represent all such choices.}\]
  and for $i,i'=1,\ldots,m$,
  \begin{equation}\label{eq: d diag S inv X}
   d_{i,i'}\overset{\mathrm{def}}{=}(-1)^{j_{i'}-1}(\sqrt{-1}e^{\sqrt{-1}\theta_0})^{j_{i'}}(SX^{-1})_{i,i'}.\quad\quad\footnote{By Corollary~\ref{cor: Xellell hermitian non-degenerate xell1ell2 vanishes}, $X$ is invertible.}
  \end{equation}
  Particularly, if $k=\sum_{\ell'>\ell}m_{\ell'}n_{\ell'}$ for some $\ell=1,\ldots,s$ (or equivalently, $\varphi(k)=\sum_{\ell'\leq\ell}m_{\ell'}$), we have that $\mathcal{I}_{k}=\left\{\{1,\ldots,\sum_{\ell'=1}^{\ell}m_{\ell'}\}\right\}$ and as $t\to 0$,
  \begin{equation}\label{eq: ck asymp entire block}
    c_k(t)=(-1)^{N-k}t^{\varphi(k)}c_{N}(0)\det\left[(d_{i,i'})_{i,i'=1,\ldots,\varphi(k)}\right]+o(t^{\varphi(k)}).
  \end{equation}
  \end{lemma}
\begin{remark}
 Lemma~\ref{lem: asymptotics coeff char polynomial} is valid without \eqref{defn: weaker definiteness assumption}. However, to get the exact order of asymptotics, we need to ensure that the determinant in \eqref{eq: ck asymp entire block} does not vanish, which follows from \eqref{defn: weaker definiteness assumption}.
\end{remark}

 \begin{proof}[Proof of Lemma~\ref{lem: asymptotics coeff char polynomial}]
 Note that \[p(\lambda,0)=\det(\lambda\cdot\Id-\gamma(0))=(\lambda-e^{\sqrt{-1}\theta_0})^{N}\prod_{\mu\neq e^{\sqrt{-1}\theta_0}}((\lambda- e^{\sqrt{-1}\theta_0})+( e^{\sqrt{-1}\theta_0}-\mu)),\]
 where $\mu$ is an eigenvalue of $\gamma(0)$. Comparing this with the expansion of $p(\lambda,0)$ at $e^{\sqrt{-1}\theta_0}$ in \eqref{defn: ck expansion at lambda0}, we conclude that $c_k(0)=0$ for $k=0,\ldots,N-1$ and $c_{N}(0)$ is given by \eqref{eq: cN0}. Next, we will estimate $c_k(t)$ for $k=0,\ldots,N-1$. We will expand $c_k(t)$ by using exterior powers of linear maps, identify and calculate the major terms. For simplicity of notation, we give the proof for $N=n$ and $e^{\sqrt{-1}\theta_0}\neq \pm 1$. The argument for the general case is quite similar. We briefly explain necessary modifications in Remark~\ref{rem: necessary modification for N neq n} and omit the details.

 In this case, we see that
  \[c_{N}(0)=(e^{\sqrt{-1}\theta_{0}}-e^{-\sqrt{-1}\theta_0})^n.\]
  Recall the definitions and \eqref{eq: char polynomial three terms} in Subsection~\ref{subsect: exterior powers of linear maps}. Note that
 \begin{equation}\label{eq: expression of ck by ck1k2}
  c_k(t)=\sum_{\substack{k_1+k_2=2n-k,\\k_1,k_2\geq 0}}(-1)^{k_2}t^{k_2}c_{k_1,k_2}(t),
 \end{equation}
 where
 \[c_{k_1,k_2}(t)=\bigwedge(2n,k_1,k_2,e^{\sqrt{-1}\theta_0}\cdot\Id-\gamma(0),\frac{1}{t}(\gamma(t)-\gamma(0))).\]
 Note that $t\mapsto c_{k_1,k_2}(t)$ is continuous and
 \[c_{k_1,k_2}(0)=\bigwedge(2n,k_1,k_2,e^{\sqrt{-1}\theta_0}\cdot\Id-\gamma(0),\frac{\mathrm{d}}{\mathrm{d} t}\gamma(0)).\]

To calculate $c_k(t)$ and $c_{k_1,k_2}(t)$, we need to fix a basis of $\mathbb{C}^{2n}$. Recall the notations given by \eqref{eq: gammaxi=lambda0xi-xi} and \eqref{eq: xiij basis}. Then, $\sum_{i=1}^{m}j_i=N=n$. By taking complex conjugates, we see that $\{\bar{\xi}_{i,j}\}_{i,j}$ is a basis of the invariant space $E_{e^{-\sqrt{-1}\theta_0}}(\gamma(0))$ associated with the eigenvalue $e^{-\sqrt{-1}\theta_0}$ of the matrix $\gamma(0)$ with properties similar to \eqref{eq: gammaxi=lambda0xi-xi}. Moreover, by Lemma~\ref{lem: G othorgonal invariant spaces} and non-degeneracy of $(\cdot,\cdot)_{G}$, $\{\xi_{i,j},\bar{\xi}_{i,j}\}_{i,j}$ is a basis of $\mathbb{C}^{2n}$.

Before we proceed with the expansion of $c_{k_1,k_2}(t)$, let us firstly fix several notations. We define $M_0=\Id$, $M_1=e^{\sqrt{-1}\theta_0}\cdot\Id-\gamma(0)$ and $M_2=\frac{1}{t}(\gamma(t)-\gamma(0))$. Let $P=\{(i,j):i=1,\ldots,m;j=1,\ldots,j_i\}$. Then, the generalized eigenvectors $\{\xi_{i,j}\}_{i,j}$ are indexed by $P$. We fix the lexicographic order on $P$ so that $P$ is totally ordered. In the definition of $c_{k_1,k_2}(t)$, for each vector $\xi_{p}$ ($p\in P$), we apply to it some linear map selected from the three different linear maps $M_0$, $M_1$ and $M_2$, and then multiply the resulting vectors via wedge products. Let $\Omega=\{0,1,2\}^{P}$. Then, the choice of linear maps is represented by an element in $\Omega$. For instance, for $\sigma=(\sigma_p)_{p\in P}\in\Omega$, for a vector $\xi_{p}$, we apply to it the map $M_{\sigma_{p}}$. For the vectors $\{\bar{\xi}_{p}\}_{p\in P}$, we use the similar notations $\tilde{\sigma}$. In the definition of $c_{k_1,k_2}(t)$, we don't sum over all possible assignment $\sigma,\tilde{\sigma}\in\Omega$. The requirement is that we use $k_1$ times the map $M_1$, $k_2$ times the map $M_2$ and $2n-k_1-k_2$ times the map $M_0$. To count the number of occurrence of a particular map $M_i$ ($i=0,1,2$), we introduce the following notation: for $\sigma\in \Omega$, a subset of indices $Q\subset P$ and $\alpha=0,1,2$, we define
 \[N_{\alpha}(\sigma,Q)=\sum_{p\in Q}1_{\sigma_p=\alpha}.\]
For $q_1+q_2\leq n$, we define
 \[\Omega_{q_1,q_2}\overset{\mathrm{def}}{=}\{\sigma\in\Omega:N_1(\sigma,P)=q_1,N_2(\sigma,P)=q_2\}.\]
Then, we express $c_{k_1,k_2}(t)$ as follows:
\begin{equation*}
\sum_{\substack{q_1+\tilde{q}_1=k_1,\\ q_2+\tilde{q}_2=k_2}}\sum_{\substack{\sigma\in\Omega_{q_1,q_2},\\
\tilde{\sigma}\in\Omega_{\tilde{q}_1,\tilde{q}_2}}}\left(\wedge_{p\in P}M_{\sigma_p}\xi_{p}\right)\wedge\left(\wedge_{p\in P}M_{\tilde{\sigma}_p}\bar{\xi}_{p}\right)=c_{k_1,k_2}(t)\left(\wedge_{p\in P}\xi_{p}\right)\wedge\left(\wedge_{p\in P}\bar{\xi}_{p}\right).
\end{equation*}
At the first sight, the above expression may seem to be impractical as it evolves lots of terms. However, not all the terms in the above summation contribute to $c_{k_1,k_2}(t)$. For instance, if we apply $M_1$ to an eigenvector associated with the eigenvalue $e^{\sqrt{-1}\theta_0}$ of $\gamma(0)$, then we immediately get a zero. The other possibility to get a zero contribution is due to the skew-symmetry of the wedge product. For instance, for an eigenvector $v_1$ and a generalized eigenvector $v_2$ such that $M_1v_2=v_1$ and $M_1v_1=0$, we see that $M_1v_2\wedge M_0v_1=0$. We will combine these two observations and give a necessary condition for non-trivial contributions. For $i=1,\ldots,m$, we define $P_i=\{(i,1),\ldots,(i,j_i)\}$ with the lexicographic order. The index set $P_i$ corresponds to the generalized eigenvectors associated with the $i$-th Jordan block. Note that for $i=1,\ldots,m$, we have that $\wedge_{p\in P_i}M_{\sigma_p}\xi_{p}=0$ if $N_2(\sigma,P_i)=0$ and $N_1(\sigma,P_i)\geq 1$. So, roughly speaking, in order that the term $\left(\wedge_{p\in P}M_{\sigma_p}\xi_{p}\right)\wedge\left(\wedge_{p\in P}M_{\tilde{\sigma}_p}\bar{\xi}_{p}\right)$ is not vanishing, the following condition is necessary: for the generalized eigenvectors corresponding to some Jordan block, if we don't apply $M_2$ to them, then we have to apply $M_0$ to all these vectors. In this sense, we need certain minimal amount of $M_0$ available. To be more precise, if the number of $M_2$ available is strictly less than the total number $m$ of the Jordan blocks associated with $e^{\sqrt{-1}\theta_0}$, then at least $m-N_2(\sigma,P)$ blocks are free of $M_2$ and we have to apply $M_0$ to all the corresponding generalized eigenvectors. The minimum of the total size of $m-N_2(\sigma,P)$ many Jordan blocks is $\sum_{i>N_2(\sigma,P)}j_i$. Hence, in order to get non-zero contribution, we need that $N_0(\sigma,P)\geq \sum_{i>N_2(\sigma,P)}j_i$. Noting that $N_2(\sigma,P)\leq k_2$ and $2n-k_1-k_2=N_0(\sigma,P)+N_0(\tilde{\sigma},P)\geq N_0(\sigma,P)$, we need that $2n-k_1-k_2\geq \sum_{i>k_2}j_i$, which is equivalent to $k_2\geq \varphi(2n-k_1-k_2)$. Hence, for $k=2n-k_1-k_2$, we have that
\begin{equation}\label{eq: ck1k2 vanishes}
 c_{k_1,k_2}(t)=0\text{ if }k_2<\varphi(k).
\end{equation}
By \eqref{eq: expression of ck by ck1k2} and \eqref{eq: ck1k2 vanishes}, for $k=0,\ldots,n-1$, as $t\to 0$,
 \begin{equation}\label{eq: expression of ck by ck1k2 with special k}
  c_k(t)=\sum_{k_2=\varphi(k)}^{2n-k}(-t)^{k_2}c_{2n-k-k_2,k_2}(t)=(-t)^{\varphi(k)}c_{2n-k-\varphi(k),\varphi(k)}(0)+o(t^{\varphi(k)}),
 \end{equation}
 which is precisely Equation~\eqref{eq: expression of ck big O}.

 Next, we will calculate $c_{2n-k-\varphi(k),\varphi(k)}(0)$ when $k=\sum_{i>\varphi(k)}j_i$. For simplicity of notation, let $K_0=e^{\sqrt{-1}\theta_0}\cdot\Id-\gamma(0)$ and $\Delta_0=\frac{\mathrm{d}\gamma}{\mathrm{d} t}(0)$. (We decide to abandon the use of notations $M_0$, $M_1$ and $M_2$ since we would like to emphasize the difference between $K_0$ and $\Delta_0$.) We have that
 \begin{equation*}
  c_{2n-k-\varphi(k),\varphi(k)}(0)=\bigwedge(2n,2n-k-\varphi(k),\varphi(k),K_0,\Delta_0),
 \end{equation*}
 which can be expanded as before. From previous discussion above \eqref{eq: ck1k2 vanishes}, to get non-zero contributions, there aren't many choices for the assignments of the maps $\Id$, $K_0$ and $\Delta_0$: for the vectors $\bar{\xi}_{i,j}$, we apply $K_0$ to them; for the generalized eigenvectors of the biggest $\varphi(k)$ Jordan blocks associated with $e^{\sqrt{-1}\theta_0}$, we apply $\Delta_0$ to each eigenvector and $K_0$ to the remainder so that we use only one $\Delta_0$ for each big Jordan blocks; for the generalized eigenvectors of the remainder small Jordan blocks associated with $e^{\sqrt{-1}\theta_0}$, we apply the map $\Id$ to them. Accordingly, we have that
  \begin{equation}\label{eq: expression of ck1k201}
   \sum_{I\in\mathcal{I}_{k}}\left(\wedge_{i=1}^{m}\omega_{i,I}\right)\wedge\left(\wedge_{p\in P}K_0\bar{\xi}_{p}\right)=c_{2n-k-\varphi(k),\varphi(k)}(0)\left(\wedge_{p\in P}\xi_{p}\right)\wedge\left(\wedge_{p\in P}\bar{\xi}_{p}\right),
 \end{equation}
 where $\mathcal{I}_k$ represents different choices of the biggest $\varphi(k)$ many Jordan blocks and
 \[\omega_{i,I}=1_{i\in I}\cdot\Delta_0\xi_{i,1}\wedge(\wedge_{j=2}^{j_i}K_0\xi_{i,j})+1_{i\notin I}\cdot\wedge_{j=1}^{j_i}\xi_{i,j}.\]
 By \eqref{eq: gammaxi=lambda0xi-xi}, for $i=1,\ldots,m$ and $j=1,\ldots,j_i$, we have that
 \[K_0\xi_{i,j}=\xi_{i,j-1}\text{ and }K_0\bar{\xi}_{i,j}=(e^{\sqrt{-1}\theta_0}-e^{-\sqrt{-1}\theta_0})\bar{\xi}_{i,j}+\bar{\xi}_{i,j-1}\]
 where $\xi_{i,0}=0$. Hence, we have that
 \begin{equation}\label{eq: expression of wedgeK0barxip}
  \wedge_{p\in P}K_0\bar{\xi}_{p}=(e^{\sqrt{-1}\theta_0}-e^{-\sqrt{-1}\theta_0})^{n}\cdot\wedge_{p\in P}\bar{\xi}_{p}=c_{N}(0)\cdot\wedge_{p\in P}\bar{\xi}_p
 \end{equation}
 and that
 \begin{equation*}
  \omega_{i,I}=1_{i\in I}\cdot\Delta_0\xi_{i,1}\wedge(\wedge_{j=1}^{j_i-1}\xi_{i,j})+1_{i\notin I}\cdot\wedge_{j=1}^{j_i}\xi_{i,j}.
 \end{equation*}
 The vector $\Delta_0\xi_{i,1}$ can be uniquely expressed as a linear combination of the basis $(\xi_{i,j},\bar{\xi}_{i,j})_{i,j}$. We denote by $\tilde{d}_{i,i'}$ the coefficient of $\Delta_0\xi_{i,1}$ before $\xi_{i',j_{i'}}$. Denote by $\mathcal{S}_{I}$ all permutations of the set $I\subset\{1,\ldots,m\}$ and by $\Sgn(g)$ the signature of a permutation $g$. Then, we have that
 \begin{align*}
  \wedge_{i=1}^{m}w_{i,I}=&\sum_{g\in\mathcal{S}_{I}}\wedge_{i=1}^{m}(1_{i\in I}\cdot(-1)^{j_i-1}\cdot\tilde{d}_{i,g(i)}\cdot(\wedge_{j=1}^{j_i-1}\xi_{i,j})\wedge \xi_{g(i),j_{g(i)}}+1_{i\notin I}\cdot\wedge_{j=1}^{j_i}\xi_{i,j})\quad (\text{mod }\wedge_{p\in P}\bar{\xi}_p)\notag\\
  =&(-1)^{\sum_{i\in I}(j_i-1)}\sum_{g\in\mathcal{S}_{I}}(-1)^{\Sgn(g)}\prod_{i\in I}\tilde{d}_{i,g(i)}\cdot\wedge_{i=1}^{m}(\wedge_{j=1}^{j_i}\xi_{i,j})\quad (\text{mod }\wedge_{p\in P}\bar{\xi}_p)\notag\\
  =&(-1)^{\sum_{i\in I}(j_i-1)}\cdot\det(\tilde{d}_{i,i'})_{i,i'\in I}\cdot\wedge_{i=1}^{m}(\wedge_{j=1}^{j_i}\xi_{i,j})\quad (\text{mod }\wedge_{p\in P}\bar{\xi}_p)
 \end{align*}
 By definition of $\mathcal{I}_{k}$, for $k=\sum_{i>\varphi(k)}j_i$ and $I\in\mathcal{I}_{k}$, we have that $\sharp I=\varphi(k)$ and $\sum_{i\in I}(j_i-1)=N-k-\varphi(k)$. Hence, we obtain that
 \begin{equation}\label{eq: expression of wedge omegaiI}
  \wedge_{i=1}^{m}w_{i,I}=(-1)^{N-k-\varphi(k)}\cdot\det(\tilde{d}_{i,i'})_{i,i'\in I}\cdot\wedge_{i=1}^{m}(\wedge_{j=1}^{j_i}\xi_{i,j})\quad (\text{mod }\wedge_{p\in P}\bar{\xi}_p).
 \end{equation}
 Next, we will show that $\tilde{d}_{i,i'}$ equals $d_{i,i'}$ defined by \eqref{eq: d diag S inv X}. On one hand, since $\Delta_0=J_{2n}A(0)\gamma(0)$, $J_{2n}^{*}J_{2n}=\Id_{2n}$ and $\gamma(0)\xi_{i,1}=e^{\sqrt{-1}\theta_0}\xi_{i,1}$, we have that
 \begin{equation}\label{eq: Delta0xii1xiip1od}
 \langle\Delta_0\xi_{i,1},J_{2n}\xi_{i',1}\rangle=\langle J_{2n}A(0)\gamma(0)\xi_{i,1},J_{2n}\xi_{i',1}\rangle=e^{\sqrt{-1}\theta_0}\langle A(0)\xi_{i,1},\xi_{i',1}\rangle=e^{\sqrt{-1}\theta_0}S_{i,i'}.
 \end{equation}
 On the other hand, by Lemmas~\ref{lem: G othorgonal invariant spaces} and \ref{lem: vector G products in same invariant space}, we see that $\langle \bar{\xi}_{i,j},J_{2n}\xi_{i',1}\rangle=0$ for all $i,i'=1,\ldots,m$ and $j=1,\ldots,j_i$, and that $\langle\xi_{i,j},J_{2n}\xi_{i',1}\rangle=0$ for all $i=1,\ldots,m$ and $j=1,\ldots,j_{i}-1$. Hence, together with the definition of the matrix $X$ given by \eqref{eq: defn X xiJxi}, we get that
 \begin{equation}\label{eq: Delta0xii1xiip1wd}
 \langle\Delta_0\xi_{i,1},J_{2n}\xi_{i',1}\rangle=\sum_{i''=1}^{m}\tilde{d}_{i,i''}\langle\xi_{i'',j_{i''}},J_{2n}\xi_{i',1}\rangle=e^{\sqrt{-1}\theta_0}\sum_{i''=1}^{m}(-1)^{1-j_{i''}}(\sqrt{-1}e^{\sqrt{-1}\theta_0})^{-j_{i''}}\tilde{d}_{i,i''}X_{i'',i'}.
 \end{equation}
 Combining \eqref{eq: Delta0xii1xiip1od} and \eqref{eq: Delta0xii1xiip1wd}, we see that the expression of $\tilde{d}$ is given by \eqref{eq: d diag S inv X}.

 Together with \eqref{eq: expression of ck1k201}, \eqref{eq: expression of wedgeK0barxip} and \eqref{eq: expression of wedge omegaiI}, we get that
 \begin{equation}\label{eq: expression of ck1k202}
  c_{2n-k-\varphi(k),\varphi(k)}(0)=c_{N}(0)(-1)^{N-k-\varphi(k)}\sum_{I\in\mathcal{I}_{k}}\det\left[(d_{i,i'})_{i,i'\in I}\right],
 \end{equation}
 where $d$ is given by \eqref{eq: d diag S inv X}. Then, \eqref{eq: expression of ck exact N-k} follows from \eqref{eq: expression of ck by ck1k2 with special k} and \eqref{eq: expression of ck1k202}.
 Particularly, when $k=\sum_{\ell'>\ell}m_{\ell'}n_{\ell'}$ for some $\ell=1,\ldots,s$, we have that $\mathcal{I}_{k}=\{\{1,\ldots,\varphi(k)\}\}$ and \eqref{eq: ck asymp entire block} follows.
\end{proof}

The above proof is written for the case $N=n$. We briefly explain the modifications for $N\neq n$ in the following remark.
\begin{remark}\label{rem: necessary modification for N neq n}
 Instead of the eigenvectors $\{\bar{\xi}_{i,j}\}_{i,j}$, for each eigenvalue $\mu\neq e^{\sqrt{-1}\theta_0}$ with algebraic multiplicity $N(\mu)$, we take generalized eigenvectors $\{\xi^{(\mu)}_k\}_{k=1,\ldots,N(\mu)}$ as $\{\xi_{i,j}\}_{i,j}$ for the eigenvalue $e^{\sqrt{-1}\theta_0}$. Then, instead of \eqref{eq: expression of wedgeK0barxip}, we have that
 \[K_0\xi^{(\mu)}_1\wedge\cdots\wedge K_0\xi^{(\mu)}_{N(\mu)}=(e^{\sqrt{-1}\theta_0}-\mu)^{N(\mu)}\xi^{(\mu)}_1\wedge\cdots\wedge \xi^{(\mu)}_{N(\mu)}.\]
 Instead of $\langle\bar{\xi}_{i,j},J_{2n}\xi_{i',1}\rangle=0$, we use the $G$-orthogonality of the invariant spaces $E_{\mu}$ and $E_{e^{\sqrt{-1}\theta_0}}$ for $\mu\neq e^{\sqrt{-1}\theta_0}$.
\end{remark}

\subsection{Proof of Theorem~\ref{thm: krein lyubarskii c1} a) from Lemma~\ref{lem: asymptotics coeff char polynomial}}\label{subsect: proof of thm KL a from lem asymp coeff}
Recall the notations introduced in \eqref{eq: gammaxi=lambda0xi-xi}, \eqref{eq: xiij basis}, \eqref{eq: xiij regroup}, \eqref{eq: defn S Cxixi}, \eqref{eq: defn X xiJxi} and \eqref{eq: SX blocks}. As $t$ varies from $0$, the continuous branching of the eigenvalue $e^{\sqrt{-1}\theta_0}$ follows from the continuity of $t\mapsto p(\lambda,t)=\det(\lambda\cdot\Id-\gamma(t))$ and Lemma~\ref{lem: continuity of roots of general polynomials}.

Next, note that $S$ is Hermitian and strictly positive definite, $X^{(\ell,\ell)}$ is Hermitian (see Corollary~\ref{cor: Xellell hermitian non-degenerate xell1ell2 vanishes}). Hence, the roots of the polynomial \eqref{eq: defn of aellp} are non-zero real numbers.

We prove the asymptotic of eigenvalues when $t>0$. The proof for $t<0$ is similar.

By Lemma~\ref{lem: dim reduction}, without loss of generality, we assume that the eigenvalues of $\gamma(t)$ are $e^{\sqrt{-1}\theta_0}$ and $e^{-\sqrt{-1}\theta_0}$. There are two possibilities: $e^{\sqrt{-1}\theta_0}\in U\setminus\mathbb{R}$ or $e^{\sqrt{-1}\theta_0}=\pm 1$. Again, the proofs in both cases are quite similar and we only present the proof for the first case, which appears to be a bit more complicated. In this case, $p(\lambda,0)=(\lambda-e^{\sqrt{-1}\theta_0})^{n}(\lambda-e^{-\sqrt{-1}\theta_0})^{n}$.

Suppose that $\lambda(t)\in\mathbb{C}$ is a root of the polynomial $p(\lambda,t)$. For $\ell=1,\ldots,s$ and $t>0$, we consider
\begin{equation}\label{eq: defn w(t)}
w_{\ell}(t)\overset{\mathrm{def}}{=}t^{-\frac{1}{n_{\ell}}}(\lambda(t)-e^{\sqrt{-1}\theta_0}).
\end{equation}
By \eqref{defn: ck expansion at lambda0}, it is a root of the polynomial $\sum_{k=0}^{2n}c_{k}(t){t}^{\frac{k}{n_{\ell}}}w^{k}$ in $w$. Since the polynomial $p$ has $2n$ roots, there are $2n$ continuous curves $t\mapsto w(t)$ for $t\neq 0$. We will show that there are exactly $n_{\ell}m_{\ell}$ many curves with non-zero limits as $t$ tends to $0$, there are exactly $\sum_{\ell<\ell'\leq s}m_{\ell'}n_{\ell'}$ many curves with the limit $0$ as $t$ tends to $0$, and the remainder tends to $\infty$ as $t$ tends to $0$. So, there are exactly $n$ curves $t\mapsto\lambda(t)$ of eigenvalues of $\gamma(t)$ tending to $e^{-\sqrt{-1}\theta_0}$ and the remainder tends to $e^{\sqrt{-1}\theta_0}$ with possibly different speeds. Roughly speaking, each Jordan block associated with $e^{\sqrt{-1}\theta_0}$ of the size $n_{\ell}$ corresponds to $n_{\ell}$ many curves of eigenvalues, these curves are exactly H\"older-$\frac{1}{n_{\ell}}$ continuous at $t=0$ and they form an $n_{\ell}$-star at $e^{\sqrt{-1}\theta_0}$.

Our task is to find the limit of \eqref{eq: defn w(t)} by applying Lemma~\ref{lem: continuity of roots of general polynomials}. Although $w(t)$ is a root of the polynomial $\sum_{k=0}^{2n}c_{k}(t){t}^{\frac{k}{n_{\ell}}}w^{k}$, we cannot apply Lemma~\ref{lem: continuity of roots of general polynomials} directly to that polynomial since it has a trivial limit $0$ as $t\to 0$. Instead, we will divide that polynomial by certain fractal powers $t^{\tau(\ell)/n_{\ell}}$ of $t$, which is ``the biggest common factor'' of $\{c_{k}(t){t}^{\frac{k}{n_{\ell}}}\}_{k}$, and obtain a new polynomial $q(w,t)$ with the same roots and a non-trivial limit as $t\to 0$. To get the exponent $\tau(\ell)/n_{\ell}$, we will use the asymptotics of $t\mapsto c_{k}(t)$ summarized in Lemma~\ref{lem: asymptotics coeff char polynomial}. By Lemma~\ref{lem: asymptotics coeff char polynomial}, for $k=0,\ldots,n$, if $k=\sum_{\ell'>\ell}m_{\ell'}n_{\ell'}+un_{\ell}$ for some $u=0,1,\ldots,m_{\ell}$, then $\varphi(k)$ defined in \eqref{defn: varphi(k)} equals $\sum_{\ell'\leq\ell}m_{\ell'}-u$ and
\[c_{k}(t)t^{\frac{k}{n_{\ell}}}=t^{\frac{\tau(\ell)}{n_{\ell}}}(-1)^{\sum_{\ell'\leq\ell}m_{\ell'}n_{\ell'}-un_{\ell}}(e^{\sqrt{-1}\theta_0}-e^{-\sqrt{-1}\theta_0})^n\sum_{I\in\mathcal{I}_{\ell,u}}\det(d_{i,i'})_{i,i'\in I}+o(t^{\tau(\ell)/n_{\ell}}),\]
where $\tau(\ell)\overset{\mathrm{def}}{=}\sum_{\ell'=1}^{s}m_{\ell'}\min\left(n_{\ell'},n_{\ell}\right)$ and \[\mathcal{I}_{\ell,u}\overset{\mathrm{def}}{=}\left\{I\subset\{1,2,\ldots,\sum_{\ell'\leq \ell}m_{\ell'}\}:\sharp I=\sum_{\ell'\leq \ell}m_{\ell'}-u,\{1,2,\ldots,\sum_{\ell'<\ell}m_{\ell'}\}\subset I\right\}.\]
Otherwise, for $k\notin \{\sum_{\ell'>\ell}m_{\ell'}n_{\ell'},\sum_{\ell'>\ell}m_{\ell'}n_{\ell'}+n_{\ell},\ldots,\sum_{\ell'>\ell}m_{\ell'}n_{\ell'}+m_{\ell}n_{\ell}\}$,
\[c_{k}(t)t^{\frac{k}{n_{\ell}}}=o(t^{\tau(\ell)/n_{\ell}})\text{ as }t\to 0.\]
Hence, we define
\begin{equation}
 q(w,t)=\sum_{k=0}^{2n}c_{k}(t)t^{\frac{k-\tau(\ell)}{n_{\ell}}}w^{k}.
\end{equation}
Note that the limiting polynomial $q(w,0)\overset{\mathrm{def}}{=}\lim_{t\to 0}q(w,t)$ exists and
\begin{equation}\label{eq: limit poly qw0 exp1}
 q(w,0)=(-1)^{\sum_{\ell'\leq\ell}m_{\ell'}n_{\ell'}}(e^{\sqrt{-1}\theta_0}-e^{-\sqrt{-1}\theta_0})^nw^{\sum_{\ell'>\ell}m_{\ell'}n_{\ell'}}\sum_{u=0}^{m_{\ell}}(-w)^{un_{\ell}}\sum_{I\in\mathcal{I}_{\ell,u}}\det(d_{i,i'})_{i,i'\in I}.
\end{equation}
We write $d$ in block matrix as $S$ and $X$ in \eqref{eq: SX blocks}, i.e.,
$d=\begin{bmatrix}
 d^{(1,1)} & \cdots & d^{(1,s)}\\
 \vdots & \ddots & \vdots\\
 d^{(s,1)} & \cdots & d^{(s,s)}
\end{bmatrix}$. (For $1\leq \ell_1,\ell_2\leq s$, we note that $d^{(\ell_1,\ell_2)}$ is an $m_{\ell_1}\times m_{\ell_2}$-matrix.) For $I\in \mathcal{I}_{\ell,u}$, $(d_{i,i'})_{i,i'\in I}$ is the square matrix obtained by deleting $u$ elements on the diagonal of $d^{(\ell,\ell)}$ together with the rows and columns containing them from the matrix $\begin{bmatrix}
 d^{(1,1)} & \cdots & d^{(1,\ell)}\\
 \vdots & \ddots & \vdots\\
 d^{(\ell,1)} & \cdots & d^{(\ell,\ell)}
\end{bmatrix}$. When we sum over $\mathcal{I}_{\ell,u}$ in \eqref{eq: limit poly qw0 exp1}, we sum over all such choices of principle minors. Hence, we see that
\begin{equation}\label{eq: limit poly qw0 exp2}
 q(w,0)=(-1)^{\sum_{\ell'\leq\ell}m_{\ell'}n_{\ell'}}(e^{\sqrt{-1}\theta_0}-e^{-\sqrt{-1}\theta_0})^nw^{\sum_{\ell'>\ell}m_{\ell'}n_{\ell'}}Q_{\ell}(w),
\end{equation}
where
\begin{equation}
Q_{\ell}(w)=\det\begin{bmatrix}
 d^{(1,1)} & \cdots & d^{(1,\ell-1)} & d^{(1,\ell)}\\
 \vdots & \ddots & \vdots & \vdots\\
 d^{(\ell-1,1)} & \cdots & d^{(\ell-1,\ell-1)} & d^{(\ell-1,\ell)}\\
 d^{(\ell,1)} & \cdots & d^{(\ell,\ell-1)} & d^{(\ell,\ell)}+(-w)^{n_{\ell}}\cdot\Id_{m_{\ell}}
 \end{bmatrix}.
\end{equation}
By expanding the determinant $Q_{\ell}(w)$ in polynomials of $w$, we find that \eqref{eq: limit poly qw0 exp1} and \eqref{eq: limit poly qw0 exp2} coincide.
Similarly to the calculation from \eqref{eq: Q(w) special case} to \eqref{eq: tilde Q(w) special case}, by the relation \eqref{eq: d diag S inv X} between the matrices $d$, $S$ and $X$ and the fact that $X$ is upper triangular in the block sense (Corollary~\ref{cor: Xellell hermitian non-degenerate xell1ell2 vanishes}), we get that $Q_{\ell}(w)=0$ iff $w$ is the root of the polynomial
\begin{equation}
  \tilde{Q}_{\ell}(w)\overset{\mathrm{def}}{=}\det\begin{bmatrix}
  S^{(1,1)} & \cdots & S^{(1,\ell-1)} & S^{(1,\ell)}\\
  \vdots & \ddots & \vdots & \vdots\\
  S^{(\ell-1,1)} & \cdots & S^{(\ell-1,\ell-1)} & S^{(\ell-1,\ell)}\\
  S^{(\ell,1)} & \cdots & S^{(\ell,\ell-1)} & S^{(\ell,\ell)}-w^{n_{\ell}}(\sqrt{-1}e^{\sqrt{-1}\theta_0})^{-n_{\ell}}X^{(\ell,\ell)}
  \end{bmatrix}.
 \end{equation}
 Hence, there are $m_{\ell}n_{\ell}$ many roots $\{\omega_{\ell,p,q}\}_{p=1,\ldots,m_{\ell};q=1,\ldots,n_{\ell}}$ such that for fixed integers $\ell$ and $p$, $\left\{\frac{\omega_{\ell,p,q}}{\sqrt{-1}e^{\sqrt{-1}\theta_0}}\right\}_{q=1,\ldots,n_{\ell}}$ are the $n_{\ell}$-th roots of $a_{\ell,p}$ with multiplicities. (Recall that $a_{\ell,p}$ are the roots of \eqref{eq: defn of aellp}.) By Lemma~\ref{lem: continuity of roots of general polynomials}, there are corresponding $w_{\ell,p,q}(t)$ and $\lambda_{\ell,p,q}(t)=e^{\sqrt{-1}\theta_0}+t^{\frac{1}{n_{\ell}}}w_{\ell,p,q}(t)$ for $p=1,\ldots,m_{\ell}$ and $q=1,\ldots,n_{\ell}$ such that $w_{\ell,p,q}(0)=\lim_{t\to 0}w_{\ell,p,q}(t)$ exists and $(w_{\ell,p,q}(0))_{p=1,\ldots,m_{\ell};q=1,\ldots,n_{\ell}}$ are roots of
 $\tilde{Q}_{\ell}(w)$. Or equivalently, \eqref{eq: asymp eigenvalues} holds.

 \begin{remark}\label{rem: weaker assumptions on A}
  During the proof of Theorem~\ref{thm: krein lyubarskii c1} a), the only purpose of assuming \eqref{defn: weaker definiteness assumption} is to ensure that $\tilde{Q}_{\ell}(w)$ has non-zero roots. Hence, Theorem~\ref{thm: krein lyubarskii c1} a) still holds under the following weaker condition:
  \begin{equation}\det\begin{bmatrix}
  S^{(1,1)} & \cdots & S^{(1,\ell)}\\
  \vdots & \ddots & \vdots\\
  S^{(\ell,1)} & \cdots & S^{(\ell,\ell)}
  \end{bmatrix}\neq 0\text{ for all }\ell=1,\ldots,s.
  \end{equation}
  Or equivalently in the following coordinate-free form: the bilinear form $\langle A(0)\cdot,\cdot\rangle$ is non-degenerate on the spaces $V_\ell$ for all integer $\ell$, where
  \begin{equation*}
   V_\ell=\ker(e^{\sqrt{-1}\theta_0}\cdot\Id-\gamma(t))\cap(e^{\sqrt{-1}\theta_0}\cdot\Id-\gamma(t))^{\ell}\left(\ker(e^{\sqrt{-1}\theta_0}\cdot\Id-\gamma(t))^{2n}\right).
  \end{equation*}
 \end{remark}

\section{Proof of Theorem~\ref{thm: krein lyubarskii c1} b)}\label{sect: proof of KL C1 b}
Our proof strategy is to approximate the continuous curve $t\mapsto A(t)$ by analytic curves. To prove Theorem~\ref{thm: krein lyubarskii c1} b), we use Theorem~\ref{thm: krein lyubarskii c1} a) proved in Section~\ref{sect: proof of KL C1 a} and Theorem~\ref{thm: krein lyubarskii c1} for the analytic case. We present a sketch of Theorem~\ref{thm: krein lyubarskii c1} when $t\mapsto A(t)$ is real analytic in Subsection~\ref{subsect: krein lyubarskii analytic}.

We choose to present the proof for $n_{\ell}$ odd, $t>0$ and $a_{\ell,p}>0$. The proofs for other cases are similar and we left them to the reader. By Theorem~\ref{thm: krein lyubarskii c1} a), we see that $(\lambda_{\ell,p,q}(t))_{q=2,\ldots,n_{\ell}}$ are outside of $U$ for sufficiently small $t$. It remains to prove that $\lambda_{\ell,p,1}(t)$ is a Krein positive definite eigenvalue on $U$. By Theorem~\ref{thm: krein lyubarskii c1} a), we have that \[\lim_{t\downarrow 0}\frac{\lambda_{\ell,p,1}(t)-e^{\sqrt{-1}\theta_0}}{\sqrt{-1}e^{\sqrt{-1}\theta_0}t^{\frac{1}{n_{\ell}}}}>0.\] Hence, as $t$ increases from $0$, tangent to the circle and counter-clockwise, $\lambda_{\ell,p,1}(t)$ continuously branches from $e^{\sqrt{-1}\theta_0}$. We need to show that $\lambda_{\ell,p,1}(t)\in U$ for sufficiently small $t$.

We define
\begin{equation*}
I_{+}=\left\{(\ell,p,q):\lim_{t\downarrow 0}\frac{\lambda_{\ell,p,q}(t)-e^{\sqrt{-1}\theta_0}}{\sqrt{-1}e^{\sqrt{-1}\theta_0}t^{\frac{1}{n_{\ell}}}}\in(0,+\infty)\right\},
\end{equation*}
\begin{equation*}
I_{-}=\left\{(\ell,p,q):\lim_{t\downarrow 0}\frac{\lambda_{\ell,p,q}(t)-e^{\sqrt{-1}\theta_0}}{\sqrt{-1}e^{\sqrt{-1}\theta_0}t^{\frac{1}{n_{\ell}}}}\in(-\infty,0)\right\},
\end{equation*}
\begin{equation*}
 J_{+}(t)=\{(\ell,p,q):\lambda_{\ell,p,q}(t)\text{ is a Krein positive definite eigenvalue on }U\},
\end{equation*}
\begin{equation*}
 J_{-}(t)=\{(\ell,p,q):\lambda_{\ell,p,q}(t)\text{ is a Krein negative definite eigenvalue on }U\},
\end{equation*}
\begin{equation*}
 K_{+}(t)=\left\{(\ell,p,q):\lambda_{\ell,p,q}(t)\in U\setminus\{e^{\sqrt{-1}\theta_0}\}\text{ and it is on the counter-clockwise side of }e^{\sqrt{-1}\theta_0}\right\},
\end{equation*}
\begin{equation*}
 K_{-}(t)=\left\{(\ell,p,q):\lambda_{\ell,p,q}(t)\in U\setminus\{e^{\sqrt{-1}\theta_0}\}\text{ and it is on the clockwise side of }e^{\sqrt{-1}\theta_0}\right\}.\quad\footnote{When $t$ is sufficiently close to $0$, $\lambda_{\ell,p,q}(t)$ locates near $e^{\sqrt{-1}\theta_0}$. Thus, it makes sense to use the notions ``counter-clockwise side'' and ``clockwise side''.}
\end{equation*}
We will show that
\begin{equation}\label{eq: coincidence of I, J, K}
 \lim_{t\downarrow 0}J_{+}(t)=\lim_{t\downarrow 0}K_{+}(t)=I_{+}\text{ and }\lim_{t\downarrow 0}J_{-}(t)=\lim_{t\downarrow 0}K_{-}(t)=I_{-}.
\end{equation}

The continuity of $t\mapsto\det(\lambda\cdot\Id-\gamma(t))$ implies the continuity of the eigenvalues as $t$ varies. Also, by the first order asymptotics in Theorem~\ref{thm: krein lyubarskii c1} a), we see that $e^{\sqrt{-1}\theta_0}$ is no longer an eigenvalue of $\gamma(t)$ if $t$ varies from $0$ a bit. Hence, there exist $r>0$ and $\delta>0$ such that for $t\in(0,\delta]$, $(\lambda_{\ell,p,q}(t))_{\ell=1,\ldots,s;p=1,\ldots,m_{\ell};q=1,\ldots,n_{\ell}}$ are located in the punctured open disk $B(e^{\sqrt{-1}\theta_0},r)\setminus\{e^{\sqrt{-1}\theta_0}\}$ centered at $e^{\sqrt{-1}\theta_0}$ with the radius $r<0.1$, and the other eigenvalues of $\gamma(t)$ stay outside of $B(e^{\sqrt{-1}\theta_0},r)$. Shrinking $\delta$ if necessary, for $t\in(0,\delta]$, for $(\ell,p,q)\in I_{+}$ (resp. $(\ell,p,q)\in I_{-}$), $\lambda_{\ell,p,q}(t)$ stays on the counter-clockwise side (resp. clockwise side) of $e^{\sqrt{-1}\theta_0}$, and for $(\ell,p,q)\notin I_{-}\cup I_{+}$, $\lambda_{\ell,p,q}(t)\notin U$. Hence, $K_{+}(t)\subset I_{+}$ and $K_{-}(t)\subset I_{-}$ for $t\in(0,\delta]$.

Next, we prove that $\lim_{t\downarrow 0}\sharp K_{+}(t)\geq\sharp I_{+}\text{ and }\lim_{t\downarrow 0}\sharp K_{-}(t)\geq\sharp I_{-}$. For that purpose, we approximate the continuous curve $t\mapsto A(t)$ by analytic curves for $t\in[-1,1]$ by using Bernstein polynomials. For positive integers $M$, we define
\[A^{(M)}(t)=\sum_{k=-M}^{M}A\left(\frac{k}{M}\right)\binom{2M}{M+k}\left(\frac{1-t}{2}\right)^{M-k}\left(\frac{1+t}{2}\right)^{M+k}.\]
As a polynomial in $t$, the function $t\mapsto A^{(M)}(t)$ is analytic. By classical results on Bernstein polynomials, for continuous $t\mapsto A(t)$, $A^{(M)}(t)$ converges to $A(t)$ as $M\to\infty$ uniformly for $t\in[-1,1]$. Hence, the corresponding solution $\gamma^{(M)}(t)$ of \eqref{eq: ham system R modified initial cont pos} (with the same initial condition) also converges to $\gamma(t)$, uniformly for $t\in[-1,1]$.

We wish to use Krein-Lyubarskii theorem for approximated analytic systems, see Subsection~\ref{subsect: krein lyubarskii analytic} for a proof in analytic case. For that purpose, we need to verify the condition \eqref{defn: weaker definiteness assumption} for large enough $M$. By taking a subsequence, we may assume that  \eqref{defn: weaker definiteness assumption} holds for each $M$ and $t\in[-1,1]$. Otherwise, if \eqref{defn: weaker definiteness assumption} is violated for infinitely many $M$, then there exist sequences $\{M_{n}\}_n$, $\{t_{n}\}_{n}$, $\{\xi_{n}\}_n$ and $\{\lambda_n\}_{n}$ such that $\lim_{n\to+\infty}M_n=+\infty$, $\{t_n\}_n$ is bounded and for all $n$, $\lambda_n\in U$, $||\xi_n||_2=1$, $\gamma^{(M_n)}(t_n)\xi_n=\lambda_n\xi_n$ and $\langle A^{(M_n)}(t_n)\xi_n,\xi_n\rangle=0$. By compactness, taking subsequence if necessary, we may further assume that $\lim_{n\to+\infty}t_n=t$, $\lim_{n\to+\infty}\xi_n=\xi$ and $\lim_{n\to+\infty}\lambda_n=\lambda$. Then, by taking the limit, we see that $||\xi||_2=1$, $\lambda\in U$, $\gamma(t)\xi=\lambda\xi$ and $\langle A(t)\xi,\xi\rangle=0$, which contradicts with the assumption \eqref{defn: weaker definiteness assumption} on the continuous curve $t\mapsto \gamma(t)$.

In the following, we assume that \eqref{defn: weaker definiteness assumption} holds for each $M$.

For approximated systems, we analogously define the notations $\{\lambda^{(M)}_{\ell,p,q}(t)\}_{\ell=1,\ldots,s;p=1,\ldots,m_{\ell};q=1,\ldots,n_{\ell}}$, $I^{(M)}_{+}$, $I^{(M)}_{-}$, $J^{(M)}_{+}(t)$, $J^{(M)}_{-}(t)$, $K^{(M)}_{+}(t)$, $K^{(M)}_{-}(t)$ and $D^{(M)}$ (see \eqref{eq: defn D ham system}). In the following, we take $M$ large enough such that $(\lambda^{(M)}_{\ell,p,q}(t))_{\ell=1,\ldots,s;p=1,\ldots,m_{\ell};q=1,\ldots,n_{\ell}}$ locate in $B(e^{\sqrt{-1}\theta_0},r)$ for $t\in(0,\delta]$. For $t\notin D^{(M)}$, we define an index
\[\nu^{(M)}_{+}(t)=\sharp (K^{(M)}_{+}(t)\cap J^{(M)}_{+}(t))-\sharp (K^{(M)}_{+}(t)\cap J^{(M)}_{-}(t)).\]
Since $(D^{(M)})^{c}$ is dense, for $t\in D^{(M)}$, we may define $\nu^{(M)}_{+}(t)=\limsup_{s\uparrow t,s\notin D^{(M)}}\nu^{(M)}_{+}(s)$. Direct approximation argument relying on the convergence $\lim_{M\to\infty}\gamma^{(M)}=\gamma$ is not sufficient to conclude the desired result. Instead, we will crucially use the following feature of $\nu^{(M)}_{+}(t)$ in the argument.
\begin{claim}\label{claim: non-decrease nu M t}
 For large enough $M$, as $t$ increases from $0$ to $\delta$, the index $\nu^{(M)}_{+}(t)$ is non-decreasing and integer-valued.
\end{claim}
We focus on the application of Claim~\ref{claim: non-decrease nu M t} and postpone its proof in the end of this section.

Since $\lim_{M\to\infty}K^{(M)}_{+}(t)=K_{+}(t)$, to prove $\lim_{t\downarrow 0}\sharp K_{+}(t)\geq\sharp I_{+}$, it suffices to show that for $M$ large enough, for all $t\in(0,\delta)$, $\sharp K^{(M)}_{+}(t)\geq I_{+}$. By upper semi-continuity\footnote{Note that $\sharp K^{(M)}_{+}(t)$ counts the multiplicity.} of $t\mapsto \sharp K^{(M)}_{+}(t)$, it suffices to show the inequality for $t$ in a dense set of $(0,\delta)$, say $(0,\delta)\setminus D^{(M)}$. By definition of $\nu^{(M)}_{+}(t)$, $\sharp K^{(M)}_{+}(t)\geq \nu^{(M)}_{+}(t)$ for $t\in (0,\delta)\setminus D^{(M)}$. Hence, it is enough to show that $\inf\{\nu^{(M)}_{+}(t):t\in(0,\delta)\}\geq \sharp I_{+}$. By Claim~\ref{claim: non-decrease nu M t}, we see that $\inf\{\nu^{(M)}_{+}(t):t\in(0,\delta)\}$ equals the right limit $\nu^{(M)}_{+}(0+)$ of $\nu^{(M)}_{+}$ at $0$. Hence, it suffices to show that $\nu^{(M)}_{+}(0+)\geq \sharp I_{+}$. Note that $\lim_{t\downarrow 0}\sharp (K^{(M)}_{+}(t)\cap J^{(M)}_{+}(t))=\sharp I^{(M)}_{+}$ and $\lim_{t\downarrow 0}\sharp (K^{(M)}_{+}(t)\cap J^{(M)}_{-}(t))=0$ by Theorem~\ref{thm: krein lyubarskii c1} in the analytic case. Moreover, by Remark~\ref{rem: quanlitative A independence}, since $\gamma^{(M)}(0)=\gamma(0)$ by construction, we have that $\lim_{M\to\infty}\sharp I^{(M)}_{+}=\sharp I_{+}$. Hence, $\nu^{(M)}_{+}(0+)$ precisely equals $\sharp I_{+}$ for $M$ large enough. Therefore, we have that
\begin{equation}\label{eq: K plus t bigger than I plus}
 \sharp K_{+}(t)\geq\sharp I_{+}\text{ for }t\in(0,\delta).
\end{equation}
and similarly, we see that $\sharp K_{-}(t)\geq \sharp I_{-}$.

Hence, together with the inclusion $K_{+}(t)\subset I_{+}$ and $K_{-}(t)\subset I_{-}$ for small enough $t>0$, we get that $K_{+}(t)=I_{+}$ and $K_{-}(t)=I_{-}$. From the argument for \eqref{eq: K plus t bigger than I plus}, for $t\in (0,\delta)$ with $\delta$ small enough, $\nu_{+}^{(M)}(t)=\sharp I_{+}$ as long as $M$ is large enough such that $\sharp K_{+}^{(M)}(t)=\sharp K_{+}(t)$.

To finish the proof of \eqref{eq: coincidence of I, J, K}, consider the invariant space $W_{+}(t)$ (resp. $W_{-}(t)$) spanned by the invariant spaces associated with the eigenvalues indexed by $K_{+}(t)$ (resp. $K_{-}(t)$), i.e., $W_{+}(t)\overset{\text{def}}{=}\sum_{(\ell,p,q)\in K_{+}(t)}E_{\lambda_{\ell,p,q}(t)}$ (resp. $W_{-}(t)\overset{\text{def}}{=}\sum_{(\ell,p,q)\in K_{-}(t)}E_{\lambda_{\ell,p,q}(t)}$). We use similar notations $W^{(M)}_{+}(t)$ and $W^{(M)}_{-}(t)$ for the approximated systems. By Lemma~\ref{lem: G othorgonal invariant spaces}, the Krein form $(\cdot,\cdot)_{G}$ is non-degenerate on these spaces. It suffices to show that the negative index of $(\cdot,\cdot)_{G}|_{W_{+}(t)}$ is zero and the positive index of $(\cdot,\cdot)_{G}|_{W_{-}(t)}$ is zero for small enough $t>0$. Again, we will use the same approximated systems, analyze the analytical systems and pass to the limit in the end. The non-degeneracy of the Krein forms is an important sufficient condition for the continuity of indices.

In the following, we will give the proof for $W_{+}(t)$. The other part is similar and is left to the reader. Note that there exists small enough $\delta>0$ such that $K_{+}^{(M)}(t)=I_{+}$ for $M$ large enough and $t\in(0,\delta]$, $K_{+}(t)=I_{+}$ for $t\in(0,\delta]$ and $t\mapsto W_{+}(t)$ is continuous\footnote{See e.g. \cite[Section~5.1, Chapter~2]{KatoMR1335452}.} for $t\in(0,\delta]$. By non-degeneracy of the Krein form on $W_{+}(t)$, the positive and negative indices are invariant for $t\in(0,\delta]$. Note that $\cup_{M\in\mathbb{N}}D^{(M)}$ is countable. Hence, by decreasing $\delta$ if necessary, we assume that $\delta\notin\cup_{M\in\mathbb{N}}D^{(M)}$. We will show that the Krein form is strictly positive definite on $W_{+}(\delta)$. Note that $\lim_{M\to\infty}K_{+}^{(M)}(\delta)=I_{+}=K_{+}(\delta)$ and hence, $\lim_{M\to\infty}W^{(M)}_{+}(\delta)=W_{+}(\delta)$ (in certain Grassmannian). Therefore, as $M\to\infty$, the positive and negative indices of the restriction of the Krein form $(\cdot,\cdot)_{G}$ on $W^{(M)}_{+}(\delta)$ converge to those of $W_{+}(\delta)$. As $\delta\notin D^{(M)}$, the positive index of $(\cdot,\cdot)_{G}|_{W^{(M)}_{+}(\delta)}$ is precisely $\sharp (K^{(M)}_{+}(\delta)\cap J^{(M)}_{+}(\delta))$, which is not less than $\nu^{(M)}_{+}(\delta)$ by definition. Recall that $\nu^{(M)}_{+}(t)$ is non-decreasing and $\lim_{t\downarrow 0}\nu^{(M)}_{+}(t)=\sharp I^{(M)}_{+}$. Hence, the positive index of $(\cdot,\cdot)_{G}|_{W^{(M)}_{+}(\delta)}$ is at least $\sharp I_{+}^{(M)}$. On the other hand, $\dim W^{(M)}_{+}(\delta)=\sharp K_{+}^{(M)}(\delta)\leq \sharp I_{+}^{(M)}$. Hence, the positive and negative index of $W^{(M)}_{+}(\delta)$ are respectively $\sharp I_{+}^{(M)}$ and $0$. Also, recall that $\lim_{M\to\infty}\sharp I^{(M)}_{+}=\sharp I_{+}$. Therefore, for $M$ sufficient large, the positive and negative index of $W^{(M)}_{+}(\delta)$ are respectively $\sharp I_{+}$ and $0$. Hence, by taking $M\to\infty$, the Krein form $(\cdot,\cdot)_{G}$ must be strictly positive definite on $W_{+}(t)$ for $t\in(0,\delta]$.

\bigskip

We finish this section by verifying Claim~\ref{claim: non-decrease nu M t}.
\begin{proof}[Proof of Claim~\ref{claim: non-decrease nu M t}]
Note that $\nu^{(M)}_{+}(t)$ is integer-valued by definition. It remains to prove its monotonicity, which follows from Theorem~\ref{thm: krein lyubarskii c1} for the analytic case.

Firstly, let us recall the definition of the index of an eigenvalue on $U$ (cf. \cite[Section~1.3]{EkelandMR1051888}). For $t_0\in\mathbb{R}$ and an eigenvalue $\lambda\in U$ of $\gamma^{(M)}(t_0)$, we will define an index $\ind^{(M)}(\lambda,t_0)$ as in \cite[Section~1.3]{EkelandMR1051888}. As $t$ varies from $t_0$, the eigenvalue $\lambda$ branches into $N$ eigenvalues. (For instance, when no bifurcation occurs, we have that $N=1$.) Among these eigenvalues we denote by $p_t$ the number of Krein positive definite eigenvalues and by $q_t$ the number of Krein negative definite eigenvalues. For $t$ close to $t_0$, $t\notin D^{(M)}$. Thus, $(p_t,q_t)$ is defined in a punctured neighborhood of $t_0$. By Corollary~5 in \cite[Section~1.3]{EkelandMR1051888}, the difference $p_t-q_t$ is locally constant near $t_0$. (Alternatively, we can deduce that from Theorem~\ref{thm: krein lyubarskii c1} in the analytic case. For instance, one can check this for each group of eigenvalues $\{\lambda_{\ell,p,q}(t)\}_{q=1,\ldots,n_{\ell}}$ forming an $n_{\ell}$-star, see \eqref{eq: asymp eigenvalues}.) The index $\ind^{(M)}(\lambda,t_0)$ is defined to be the integer $p_t-q_t$ for $t$ close to $t_0$. For a Krein positive definite eigenvalue, its index is simply its algebraic (and geometric) multiplicity. For a Krein negative definite eigenvalue, the index is the opposite of its algebraic (and geometric) multiplicity. Hence, if an eigenvalue $\lambda$ branches into several ones, the sum of the indices of the eigenvalues branched from $\lambda$ must equal to the index of $\lambda$.

Note that $\nu^{(M)}_{+}(t)=\sum_{(\ell,p,q)\in K^{(M)}_{+}(t)}\ind^{(M)}(\lambda_{\ell,p,q},t)$, i.e., it is the sum of the indices of eigenvalues indexed by $K_{+}(t)$. Recall that the eigenvalues $\lambda(t)$ branched from $e^{\sqrt{-1}\theta_0}$ are located in a small disk $B(e^{\sqrt{-1}\theta_0},r)$ for $t\in(0,\delta]$. In the following, we assume that $M$ is sufficient large such that $\gamma^{(M)}(t)$ has no eigenvalue on the boundary of $B(e^{\sqrt{-1}\theta_0},r)$ for $t\in(0,\delta]$. The part of $U$ inside $B(e^{\sqrt{-1}\theta_0},r)$ is an arc with a mid-point at $e^{\sqrt{-1}\theta_0}$. The point $e^{\sqrt{-1}\theta_0}$ separates the arc into two smaller arcs. We denote by $\mathrm{arc}_{+}$ the open half arc on the counter-clockwise side of $e^{\sqrt{-1}\theta_0}$, see Figure~\ref{fig: counter-clockwise arc}.
\begin{figure}
\centering
\label{fig: counter-clockwise arc}
\includegraphics[width=0.4\textwidth]{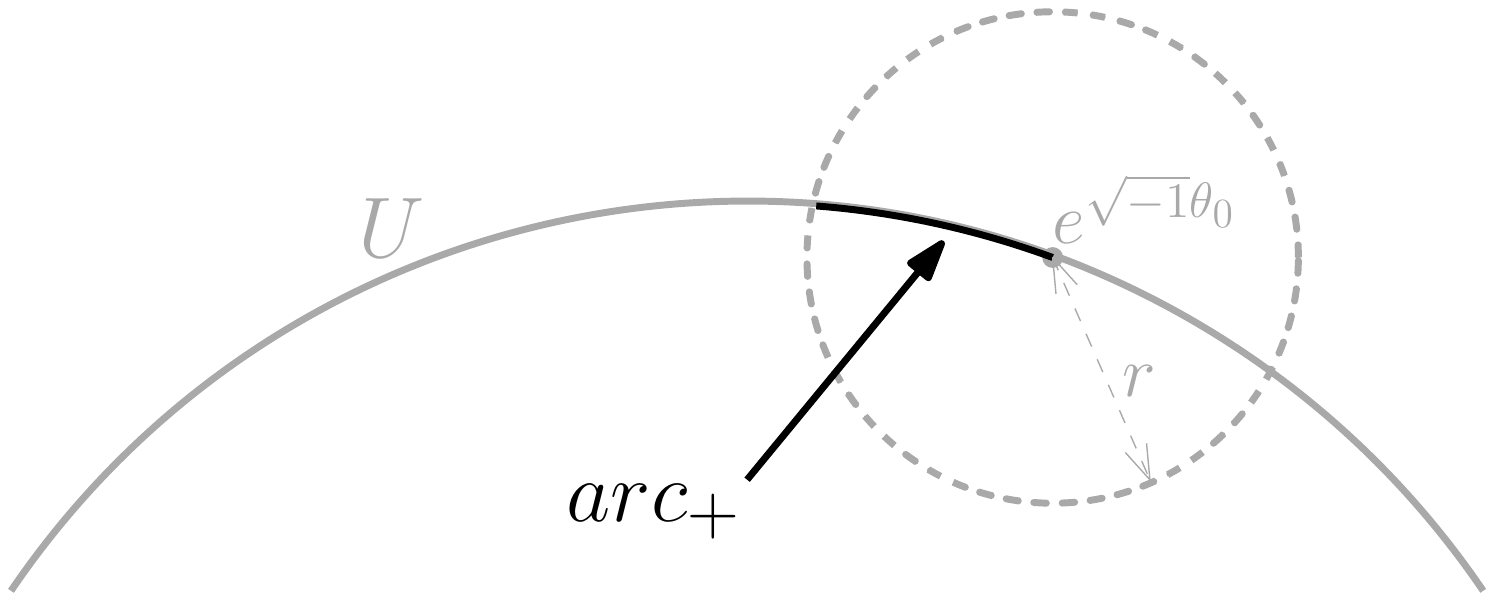}
\caption{$\mathrm{arc}_{+}$}
\end{figure}
Then, for $t\in(0,\delta)$, $\nu^{(M)}_{+}(t)$ is the sum of indices of eigenvalues in the interior of $\mathrm{arc}_{+}$. By the local constancy on the sum of the indices of branched eigenvalues, we see that $\nu^{(M)}_{+}(t)$ doesn't vary around $t_0\in(0,\delta)$ except that $\gamma^{(M)}(t_0)$ has an eigenvalue on the boundary of $\mathrm{arc}_{+}$. In the exceptional case, $\gamma^{(M)}(t_0)$ has no eigenvalue on the boundary of the disk $B(e^{\sqrt{-1}\theta_0},r)$ and $e^{\sqrt{-1}\theta_0}$ is an eigenvalue of $\gamma^{(M)}(t_0)$. By Theorem~\ref{thm: krein lyubarskii c1} for the analytic case, when $t$ increases through $t_0$, the eigenvalues entered in $\mathrm{arc}_{+}$ from $e^{\sqrt{-1}\theta_0}$ must move counter-clockwise and be Krein positive definite, the eigenvalues left $\mathrm{arc}_{+}$ from $e^{\sqrt{-1}\theta_0}$ must move clockwise and be Krein negative definite. Hence, $\nu^{(M)}_{+}(t)$ strictly increases in this case. Thus, we see that $t\mapsto \nu^{(M)}_{+}(t)$ is non-deceasing for $t\in(0,\delta)$ for $M$ sufficient large.
\end{proof}


\appendix
\section{Appendix}

 \subsection{Alternative expression for \texorpdfstring{$C(t,\varepsilon)$}{C(t,ε)}}\label{subsect: lem expression of C}
We verify the second equality in \eqref{defn: CTepsilon}.
\begin{lemma}\label{lem: expression of C}
  Let $C(t,\varepsilon)\overset{\mathrm{def}}{=}-\gamma(t,\varepsilon)^{T}J_{2n}\frac{\partial}{\partial \varepsilon}\gamma(t,\varepsilon)$. Then,
  \[C(t,\varepsilon)=\int_{0}^{t}\gamma(u,\varepsilon)^{T}\frac{\partial }{\partial \varepsilon}A(u,\varepsilon)\gamma(u,\varepsilon)\,\mathrm{d}t.\]
 \end{lemma}
  \begin{proof}
  Note that for all $\varepsilon$, $\gamma(0,\varepsilon)=\Id$. Hence, $\frac{\partial}{\partial \varepsilon}\gamma(0,\varepsilon)=0$ and $C(0,\varepsilon)=0$. Thus, it remains to show that
  \begin{equation*}
   \frac{\partial }{\partial t}C(t,\varepsilon)=\gamma(t,\varepsilon)^{T}\frac{\partial }{\partial \varepsilon}A(t,\varepsilon)\gamma(t,\varepsilon).
  \end{equation*}
  By a standard contraction argument with Gr\"{o}nwell's inequality, we have that
  \begin{equation}\label{eq: partial t partial epsilon gamma}
   \frac{\partial}{\partial t}\frac{\partial }{\partial \varepsilon}\gamma(t,\varepsilon)=J_{2n}\frac{\partial}{\partial \varepsilon}A(t,\varepsilon)\gamma(t,\varepsilon)+J_{2n}A(t,\varepsilon)\frac{\partial }{\partial \varepsilon}\gamma(t,\varepsilon).
  \end{equation}
  By \eqref{eq: perturbed ham system R} and the symmetry of $A$, we have that
  \begin{equation}\label{eq: partial t gamma star}
   \frac{\partial }{\partial t}\gamma(t,\varepsilon)^{T}=(J_{2n}A(t,\varepsilon)\gamma(t,\varepsilon))^{T}=\gamma(t,\varepsilon)^{T}A(t,\varepsilon)J_{2n}^{T}.
  \end{equation}
  Hence, combining the definition of $C$, \eqref{eq: partial t partial epsilon gamma} and \eqref{eq: partial t gamma star}, we obtain that
  \begin{align*}
   \frac{\partial}{\partial t} C(t,\varepsilon)=&-\frac{\partial }{\partial t}\gamma(t,\varepsilon)^{T}J_{2n}\frac{\partial}{\partial \varepsilon}\gamma(t,\varepsilon)-\gamma(t,\varepsilon)^{T}J_{2n}\frac{\partial}{\partial t}\frac{\partial }{\partial \varepsilon}\gamma(t,\varepsilon)\notag\\
   =&\gamma(t,\varepsilon)^{T}\frac{\partial}{\partial \varepsilon}A(t,\varepsilon)\gamma(t,\varepsilon),
  \end{align*}
  which completes the proof.
  \end{proof}

\subsection{Dimension reduction}
The following lemma helps to simplify certain notations and proofs (since it allows us to focus on one eigenvalue and to reduce the dimension in many cases). Besides, it is of independent interest. Therefore, we choose to present it here.

\begin{lemma}\label{lem: dim reduction}
 For all $n\geq2$, let $\Lambda(t_0)$ and $\tilde{\Lambda}(t_0)$ be a division of the eigenvalues of $\gamma(t_0)$ for $t_0\in\mathbb{R}$, where $t\mapsto\gamma(t)$ is the solution of \eqref{eq: ham system R modified initial cont pos}. Assume that $\Lambda(t_0)$ is closed under the conjugation $\lambda\rightarrow\bar{\lambda}$ and the circular reflection $\lambda\mapsto\bar{\lambda}^{-1}$ with respect to $U$. There exists $\varepsilon>0$ such that for $t\in[t_0-\varepsilon,t_0+\varepsilon]$, there exists a division of the eigenvalues of $\gamma(t)$ into $\Lambda(t)$ and $\tilde{\Lambda}(t)$ such that $\Lambda(t)$ is closed under the conjugation $\lambda\mapsto\bar{\lambda}$ and the circular reflection $\lambda\mapsto\bar{\lambda}^{-1}$, and $\Lambda(t)$ (resp. $\tilde{\Lambda}(t)$) converges to $\Lambda(t_0)$ (resp. $\tilde{\Lambda}(t_0)$) as $t$ tends to $t_0$. Denote by $E_{t}$ (resp. $\tilde{E}_t$) the sum of invariant spaces $(E_{\lambda})_{\lambda\in\Lambda(t)}$ (resp. $(E_{\lambda})_{\lambda\in\tilde{\Lambda}(t)}$). Then, by decreasing $\varepsilon$ if necessary, we also require that $\dim(E_t)=\dim(E_{t_0})$, $\dim(\tilde{E}_t)=\dim(\tilde{E}_{t_0})$ for $t\in[t_0-\varepsilon,t_0+\varepsilon]$ and $\lim_{t\rightarrow t_0}E_t=E_{t_0}$. Moreover, there exists a $C^{1}$ curve $t\mapsto Q(t)\in M_{2n\times 2k}(\mathbb{R})$ where $2k=\dim(E_{t_0})$ such that
  \begin{itemize}
   \item the column vectors of $Q(t)$ form a basis of $E_t$ and $Q^{*}(t)J_{2n}Q(t)=J_{2k}$, i.e., the column vectors of $Q$ form a symplectic basis of $E_t$,
   \item $\gamma(t)Q(t)=Q(t)M_{Q}(t)$ uniquely determines a $C^1$ curve $t\mapsto M_Q(t)\in\Sp(2k,\mathbb{R})$,
   \item $\mathrm{d}M_Q/\mathrm{d}t=J_{2k}Q^{*}(t)A(t)Q(t)M_Q(t)$.
  \end{itemize}
\end{lemma}
\begin{remark}
Note that the eigenvalues of $M_Q(t)$ are precisely those in $\Lambda(t)$.
\end{remark}
\begin{remark}
 Under the assumption of Lemma~\ref{lem: dim reduction}, similar to $Q(t)$ and $M_{Q}(t)$, we may take $\tilde{Q}(t)$ and $M_{\tilde{Q}}(t)$ for $\tilde{\Lambda}(t)$ and $\tilde{E}_t$. Write $Q(t)$ into two $2n\times k$ blocks: $Q(t)=\begin{pmatrix} Q_1(t) & Q_2(t)\end{pmatrix}$. Similarly, we write $\tilde{Q}(t)=\begin{pmatrix}\tilde{Q}_1(t) & \tilde{Q}_2(t)\end{pmatrix}$. Define $Y(t)=\begin{pmatrix}Q_1(t) & \tilde{Q}_1(t) & Q_2(t) & \tilde{Q}_2(t)\end{pmatrix}$. Then, $Y(t)\in\Sp(2n,\mathbb{R})$ and $\gamma(t)Y(t)=Y(t)(M_{Q}(t)\diamond M_{\tilde{Q}}(t))$, where ``$\diamond$'' denotes the symplectic summation (cf. \cite{LongMR1674313,LongMR1898560}). To be more precise, we write $M_{Q}(t)=\begin{pmatrix}
 M_{Q}^{11}(t) & M_{Q}^{12}(t)\\
 M_{Q}^{21}(t) & M_{Q}^{22}(t)
 \end{pmatrix}$, where the four sub-matrices are of equal size. We divide $M_{\tilde{Q}}(t)$ in a similar way. The symplectic sum of $M_Q(t)$ and $M_{\tilde{Q}}(t)$ is defined to be the square matrix
 \[\begin{pmatrix}
 M_{Q}^{11}(t) & 0 & M_{Q}^{12}(t) & 0\\
 0 & M_{\tilde{Q}}^{11}(t) & 0 & M_{\tilde{Q}}^{12}(t)\\
 M_{Q}^{21}(t) & 0 & M_{Q}^{22}(t) & 0\\
 0 & M_{\tilde{Q}}^{21}(t) & 0 & M_{\tilde{Q}}^{22}(t)
 \end{pmatrix}.\]
 Then, the original system is decomposed into two sub-systems. Moreover, these two sub-systems satisfy \eqref{defn: weaker definiteness assumption} if the original system satisfies such condition.
\end{remark}
\begin{proof}[Proof of Lemma~\ref{lem: dim reduction}]
 Since $\Lambda(t)$ is closed under conjugation, we have that $E_t=\mathbb{C}\otimes(\mathbb{R}^{2n}\cap E_t)$. In this sense, $E_t\subset \mathbb{R}^{2n}$ and we replace $E_t$ by $E_t\cap\mathbb{R}^{2n}$ in the following context. By continuity, there exists $\varepsilon>0$ such that for $t\in[t_0-\varepsilon,t_0+\varepsilon]$, there exists a simple smooth curve $\Gamma$ surrounding all $\Lambda(t)$ and separating $\Lambda(t)$ from $\tilde{\Lambda}(t)$. Then, we may take $P(t)=\frac{1}{2\pi\sqrt{-1}}\int_{\Gamma}(z\cdot\Id-\gamma(t))^{-1}\,\mathrm{d}z$, which projects $\mathbb{R}^{2n}$ onto $E_t$, see e.g. \cite[Section~1.4, Chapter~2]{KatoMR1335452}. Note that $E_{t_0}$ is a symplectic subspace. We choose a symplectic basis $(\xi_1,\ldots,\xi_k,\eta_1,\ldots,\eta_k)$ of $E_{t_0}$ such that $\langle\xi_i,J_{2n}\eta_j\rangle=1_{i=j}$ and $\langle\xi_i,J_{2n}\xi_j\rangle=\langle\eta_i,\eta_j\rangle=0$ for $i,j=1,\ldots,k$. Decreasing $\varepsilon$ if necessary, $(P(t)\xi_1,\ldots,P(t)\xi_k,P(t)\eta_1,\ldots,P(t)\eta_k)$ is a linear basis for $E_t$ for $t\in[t_0-\varepsilon,t_0+\varepsilon]$. However, it is in general no longer a symplectic basis. Nevertheless, by shrinking $\varepsilon$ if necessary, after Gram-Schmidt operation, we obtain a time dependent symplectic basis of $E_t$, which forms a $2n\times 2k$ matrix $T(t)$. Note that $t\mapsto T(t)$ is continuously differentiable and that
 \begin{equation}\label{eq: TstarJTJ}
   T^{*}J_{2n}T=J_{2k}.
  \end{equation}
  In general, we should not take $Q=T$. We consider the following ODE where the solution corresponds to a dynamic change of sympletic basis:
  \begin{equation}\label{eq: change of basis by solving ODE}
   \frac{\mathrm{d}V}{\mathrm{d}t}=J_{2k}T^{*}J_{2n}\frac{\mathrm{d}T}{\mathrm{d}t}V,\quad V(t_0)=\Id.
  \end{equation}
 By differentiating both sides of \eqref{eq: TstarJTJ}, we get that $T^{*}J_{2n}\frac{\mathrm{d}T}{\mathrm{d}t}$ is self-adjoint, which implies that $t\mapsto V(t)$ is a sympletic path, see e.g. \cite[Prop. 3, Section 1, Chapter 1]{EkelandMR1051888}. We define $Q\overset{\text{def}}{=}TV$. By sympleticity of $V$ and \eqref{eq: TstarJTJ}, we see that $Q(t)^{*}J_{2n}Q(t)=J_{2k}$. Also, the equation
  \begin{equation}\label{eq: RQ=QM}
   \gamma(t)Q(t)=Q(t)M_{Q}(t)
  \end{equation}
  uniquely determines a $C^1$ curve $t\mapsto M_Q(t)\in\Sp(2k,\mathbb{R})$. Indeed, by multiplying $Q(t)^{*}J_{2n}$ on both sides of \eqref{eq: RQ=QM}, we obtain that $M_{Q}(t)=-J_{2k}Q(t)^{*}J_{2n}\gamma(t)Q(t)$. By taking the derivatives and using \eqref{eq: RQ=QM}, we obtain that
  \begin{equation}\label{eq: dM/dt=JBM}
   \frac{\mathrm{d}M_{Q}}{\mathrm{d}t}=J_{2k}B_{Q}M_{Q},
  \end{equation}
  where
  \begin{equation}\label{eq: BQ expression}
   B_{Q}=Q^{*}AQ-\left(\frac{\mathrm{d}Q}{\mathrm{d}t}\right)^{*}J_{2n}Q-Q^{*}J_{2n}\gamma\frac{\mathrm{d}Q}{\mathrm{d}t}M_{Q}^{-1}.
  \end{equation}
  By $Q=TV$ and \eqref{eq: change of basis by solving ODE}, we get that
  \begin{equation}\label{eq: dQ/dt}
   \frac{\mathrm{d}Q}{\mathrm{d}t}=\frac{\mathrm{d}T}{\mathrm{d}t}V+TJ_{2k}T^{*}J_{2n}\frac{\mathrm{d}T}{\mathrm{d}t}V.
  \end{equation}
  Hence, together with \eqref{eq: TstarJTJ}, $Q=TV$ and $J_{2n}+J_{2n}^{*}=0$, we get that
  \begin{equation}\label{eq: dQ/dtstarJQ=0}
   \left(\frac{\mathrm{d}Q}{\mathrm{d}t}\right)^{*}J_{2n}Q=V^{*}\left(\frac{\mathrm{d}T}{\mathrm{d}t}\right)^{*}J_{2n}TV+V^{*}\left(\frac{\mathrm{d}T}{\mathrm{d}t}\right)^{*}J_{2n}^{*}TJ_{2k}^{*}T^{*}J_{2n}TV=0.
  \end{equation}
  It remains to prove that
  \begin{equation}\label{eq: QstarJRdQ/dtMinverse=0}
   Q^{*}J_{2n}\gamma\frac{\mathrm{d}Q}{\mathrm{d}t}M_{Q}^{-1}=0.
  \end{equation}
  By multiplying $(V^{*})^{-1}$ on the left and $M_QV^{-1}$ on the right, using $Q=TV$ and \eqref{eq: dQ/dt}, we find that \eqref{eq: QstarJRdQ/dtMinverse=0} is equivalent to
  \[(T^{*}J_{2n}\gamma+T^{*}J_{2n}\gamma TJ_{2k}T^{*}J_{2n})\frac{\mathrm{d}T}{\mathrm{d}t}=0.\]
  It would be sufficient to prove that
  \[T^{*}J_{2n}\gamma TJ_{2k}T^{*}J_{2n}=-T^{*}J_{2n}\gamma.\]
  Note that $\gamma T=TM_{T}$ uniquely determines a sympletic $2k\times 2k$ matrix $M_{T}$ since $T^{*}J_{2n}T=J_{2k}$ and $\gamma$ is sympletic. Indeed, we have that $M_{T}^{*}J_{2k}M_{T}=M_{T}^{*}T^{*}J_{2n}TM_{T}=T^{*}\gamma^{*}J_{2n}\gamma T=T^{*}J_{2n}T=J_{2k}$. By symplecity of $M_{T}$, we have that
  \begin{equation}\label{eq: RTJTstarRstar=TJTstar}
   \gamma TJ_{2k}T^{*}\gamma^{*}=TM_{T}J_{2k}M_{T}^{*}T^{*}=TJ_{2k}T^{*}.
  \end{equation}
  By writing $J_{2n}$ as $\gamma^{*}J_{2n}\gamma$, using \eqref{eq: RTJTstarRstar=TJTstar} and \eqref{eq: TstarJTJ}, we see that
  \[T^{*}J_{2n}\gamma TJ_{2k}T^{*}J_{2n}=T^{*}J_{2n}\gamma TJ_{2k}T^{*}\gamma^{*}J_{2n}\gamma=T^{*}J_{2n}TJ_{2k}T^{*}J_{2n}\gamma=-T^{*}J_{2n}\gamma,\]
  which implies \eqref{eq: QstarJRdQ/dtMinverse=0}. By \eqref{eq: dM/dt=JBM}, \eqref{eq: BQ expression}, \eqref{eq: dQ/dtstarJQ=0} and \eqref{eq: QstarJRdQ/dtMinverse=0}, we get that $\frac{\mathrm{d}M_{Q}}{\mathrm{d}t}=J_{2k}Q^{*}AQM_{Q}$.
\end{proof}

\subsection{Analytic Krein-Lyubarskii theorem}\label{subsect: krein lyubarskii analytic}
In this subsection, we provide a proof of Theorem~\ref{thm: krein lyubarskii c1} when $t\mapsto A(t)$ is real analytic. We partially follow the argument in \cite{YakubovichStarzhinskiiMR0364740} for \eqref{eq: perturbed ham system R} when $\varepsilon\mapsto A(t,\varepsilon)$ is affine in $\varepsilon$. The connection between the first order asymptotics of the eigenvalues and the Jordan structure has already been established in Section~\ref{sect: proof of KL C1 a}. We will only prove the analyticity of the eigenvalues as $t$ varies and the part b) of Theorem~\ref{thm: krein lyubarskii c1}.

By analytic continuation, the real parameter $t$ of \eqref{eq: ham system R modified initial cont pos} is extended in complex parameter $z\in\mathbb{C}$ around $0$:
\begin{equation}\label{eq: complex time linear ham}
  \frac{\mathrm{d}}{\mathrm{d}z}\gamma(z)=J_{2n}A(z)\gamma(z),\gamma(0)\in\Sp(2n,\mathbb{R}).
 \end{equation}
By analyticity of $z\mapsto A(z)$, $z\mapsto \gamma(z)$ is analytic. Since the zero set of an analytic function is isolated, the following two equations are extended to complex $z$: $A^{T}(z)=A(z)$ and $\gamma(z)^{T}J_{2n}\gamma(z)=J_{2n}$.

In \cite{YakubovichStarzhinskiiMR0364740}, they crucially used the key feature of the system that when $\gamma(z)$ has eigenvalue $\omega$ on $U$, the parameter $z$ has to be real. (Roughly speaking, the reason is that $z$ happens to be the eigenvalue of a self-adjoint operator when $\omega\in U$.) Such a phenomenon also appears for our general system \eqref{eq: complex time linear ham}, as stated in the following lemma.
\begin{lemma}\label{lem: eig on U real z}
 Consider the ODE \eqref{eq: complex time linear ham}. We assume that $z\mapsto\gamma(z)$ is analytic (or equivalently, $z\mapsto A(z)$ is analytic), $A(t)$ is real symmetric for $t\in\mathbb{R}$ and for any eigenvector $\xi$ of $\gamma(0)$ associated with an eigenvalue on $U$, $\langle A(0)\xi,\xi\rangle>0$. Then, there exists $\delta>0$, for all $z\in\mathbb{C}\setminus\mathbb{R}$ and $|z|<\delta$, $\gamma(z)$ has no eigenvalue on $U$.
\end{lemma}
To prove Lemma~\ref{lem: eig on U real z}, we need to modify the argument in \cite[Section~4.1]{YakubovichStarzhinskiiMR0364740}.
\begin{proof}[Proof of Lemma~\ref{lem: eig on U real z}]
It suffices to prove the following cannot happen:
 there exist non-real complex numbers $z_n$ tending to $0$ such that for each $z_n$, $\gamma(z_n)$ has an eigenvector $\xi_n$ with $||\xi_n||_2=1$ associated with some eigenvalue $\lambda_n\in U$. We write $z_n$ in polar coordinate as $r_ne^{\sqrt{-1}\theta_n}$ with $r_n>0$ and $\theta_n\in(-\pi,0)\cup(0,\pi)$. By taking subsequence if necessary, we assume that $\lim_{n\to+\infty}\lambda_n=\lambda$, $\lim_{n\to+\infty}\xi_n=\xi$ and $\lim_{n\to+\infty}\theta_n=\theta$.

 We expand $A(z)$ in Taylor series as $\sum_{j\geq 0}z^jA^{(j)}$ around $0$. Since $A(t)$ is real symmetric for $t\in\mathbb{R}$, $A^{(j)}$ are real symmetric for all $j\geq 0$. For $r\geq 0$ and $\theta\in\mathbb{R}$, we define $A_1(re^{\sqrt{-1}\theta})=\sum_{j\geq 0}\cos(j\theta)r^{j}A^{(j)}$ and $A_2(re^{\sqrt{-1}\theta})=\sum_{j\geq 1}\sin(j\theta)r^{j}A^{(j)}$. Then, $A_{1}(z)$ and $A_2(z)$ are real symmetric matrices and $A(z)=A_{1}(z)+\sqrt{-1}A_{2}(z)$. Moreover, there exists $C=C(A)<\infty$ such that for all $r\in[0,C^{-1})$ and $\theta\in[-\pi,\pi]$, for all $\xi\in\mathbb{C}^{2n}$ with $||\xi||_2=1$,
\begin{equation}\label{eq: bound A2}
 \left|\langle A_2(re^{\sqrt{-1}\theta})\xi,\xi\rangle\right|\leq C\cdot r\cdot |\sin(\theta)|,
\end{equation}
where $\langle\cdot,\cdot\rangle$ denotes the standard inner product on $\mathbb{C}^{2n}$, which is linear in the first vector.

For $\rho\in U$, we denote by $\mathfrak{X}(\gamma(0),\rho)$ the space of analytic paths $y:[0,1]\to\mathbb{C}^{2n}$ with the boundary condition $\gamma(0)y(1)=\rho y(0)$. Define three functions $L_0$, $L_1$ and $L_2$ on $\cup_{\rho\in U}\mathfrak{X}(\gamma(0),\rho)\times \mathfrak{X}(\gamma(0),\rho)$ as follows: for $y_1,y_2\in \mathfrak{X}(\gamma(0),\rho)$,
\begin{align*}
 L_0(y_1,y_2)=&\int_{0}^{1}\left\langle J_{2n}\frac{\mathrm{d}}{\mathrm{d}s}y_1(s),y_2(s)\right\rangle\,\mathrm{d}s\\
 L_{1,z}(y_1,y_2)=&\int_{0}^{1}\left\langle A_1(sz)y_1(s),y_2(s)\right\rangle\,\mathrm{d}s,\\
 L_{2,z}(y_1,y_2)=&\int_{0}^{1}\left\langle A_2(sz)y_1(s),y_2(s)\right\rangle\,\mathrm{d}s.
\end{align*}
Note that $L_0(y_1,y_2)=\overline{L_0(y_2,y_1)}$, $L_{1,z}(y_1,y_2)=\overline{L_{1,z}(y_2,y_1)}$ and $L_{2,z}(y_1,y_2)=\overline{L_{2,z}(y_2,y_1)}$. Hence, $L_0(y,y), L_{1,z}(y,y), L_{2,z}(y,y)\in\mathbb{R}$ for $y\in\mathfrak{X}(\gamma(0),\rho)$.

Define $x_{n}(s)=\gamma(sz_n)\xi_n$ for $s\in[0,1]$. Then, $x_n\in \cup_{\rho\in U}\mathfrak{X}(\gamma(0),\rho)$. By \eqref{eq: complex time linear ham}, we have that
\begin{equation}
 L_0(x_{n},x_{n})+z_n(L_{1,z_n}(x_{n},x_{n})+\sqrt{-1}L_{2,z_n}(x_{n},x_{n}))=0.
\end{equation}
Necessarily, the argument $\theta_n$ of $z_n$ and the argument $\psi_n$ of the complex number $L_{1,z_n}(x_{n},x_{n})+\sqrt{-1}L_{2,z_n}(x_{n},x_{n})$ differ by a multiple of $\pi$, or equivalently,
\begin{equation}\label{eq: argument inquality sin}
|\sin(\theta_n)|=|\sin(\psi_n)|.
\end{equation}
By \eqref{eq: bound A2}, there exists $C=C(A)<\infty$ such that for large enough $n$,
\[L_{2,z_n}(x_{n},x_{n})\leq C\cdot r_n|\sin(\theta_n)|.\]
By continuity, $\lim_{n\to+\infty}L_{1,z_n}(x_n,x_n)=\langle A(0)\xi,\xi\rangle>0$. Hence, as $n\to+\infty$, $\psi_n$ and $|\sin(\psi_n)|$ are of the order $r_n|\sin(\theta_n)|$, which contradicts with \eqref{eq: argument inquality sin} since $\lim_{n\to+\infty}r_n=0$.
\end{proof}
Consider the characteristic polynomial $p(\lambda,z)=\det(\lambda\cdot\Id-\gamma(z))$. Assume that $\lambda_0=e^{\sqrt{-1}\theta_0}\in U$ is an eigenvalue of $\gamma(0)$. By Weierstrass's preparation theorem of the local form of analytic functions in multi-variables, there exist integers $\ell$ and $M$ such that for $(\lambda,z)$ close to $(\lambda_0,0)$, we have that
\begin{equation*}
 p(\lambda,z)=(\lambda-\lambda_0)^{\ell}(z^{M}+a_{M-1}(\lambda)z^{M-1}+\cdots+a_{0}(\lambda))b(\lambda,z),
\end{equation*}
where $b(\lambda,z)$ is non-zero and analytic, $\{a_{i}\}_{i=0,\ldots,M-1}$ are analytic in $\lambda$ and vanish at $\lambda_0$. Note that $\ell=0$ and hence,
\begin{equation}\label{eq: weierstrass charateristic polynomial local form}
 p(\lambda,z)=(z^{M}+a_{M-1}(\lambda)z^{M-1}+\cdots+a_{0}(\lambda))b(\lambda,z).
\end{equation}
(Otherwise, $\lambda_0$ is an eigenvalue of $\gamma(z)$ as long as $z$ is sufficient close to $0$, which contradicts with Lemma~\ref{lem: eig on U real z}. Alternatively, we could see that from the first order asymptotics in Theorem~\ref{thm: krein lyubarskii c1} a) proved in Section~\ref{sect: proof of KL C1 a}. Or simply follow the argument of \cite[Proposition~2, Section~3, Chapter~1]{EkelandMR1051888}.) The solution of $p(\lambda,z)=0$ coincides with the solution of $z^{M}+a_{M-1}(\lambda)z^{M-1}+\cdots+a_0(\lambda)=0$, which is the union of the graphs of several multi-valued analytic functions $\{z_{i}(\lambda)\}_{i=1,\ldots,\tilde{M}}$ ($\tilde{M}\leq M$) in $\lambda$. By Lemma~\ref{lem: eig on U real z}, when $\lambda$ is on $U$, $z_i(\lambda)$ must lie on $\mathbb{R}$. This forces that each $z_i$ is actually single-valued analytic functions and $\tilde{M}=M$, see the lemma in \cite[Section~1.5, Chapter~3]{YakubovichStarzhinskiiMR0364740}. Hence,
\begin{equation}\label{eq: factorization of zM aM-1ZM-1 a0}
 z^{M}+a_{M-1}(\lambda)z^{M-1}+\cdots+a_0(\lambda)=\prod_{i=1}^{M}(z-z_i(\lambda)).
\end{equation}
Let
\begin{equation}\label{eq: expansion of zjlambda}
z_i(\lambda)=\sum_{k\geq j_i}e_{i,k}(\lambda-\lambda_0)^{k},\quad e_{i,j_i}\neq 0
\end{equation}
be the Taylor expansion of $z_i(\lambda)$. Inverting that expansion, we see that
$\lambda=\lambda_0+h_{i}(z^{\frac{1}{j_i}})$, where $h_i$ is analytic, $h_i(0)=0$ and $h_i'(0)\neq 0$. Note that $\lambda-\lambda_0\overset{z\to 0}{\sim} h_{i}'(0)z^{\frac{1}{j_i}}$. Compared with \eqref{eq: asymp eigenvalues}, we need to show that $\{j_i\}_{i=1,\ldots,M}$ are exactly the sizes of Jordan blocks of $\gamma(0)$ associated with $\lambda_0$. Then, $\{h_i'(0)\}_{i=1,\ldots,M}$ will be given by \eqref{eq: asymp eigenvalues}.

Firstly, let us show that $M$ is precisely the number of Jordan blocks associated with $\lambda_0$. By multiplying the first order asymptotics of the eigenvalues in \eqref{eq: asymp eigenvalues}, we see that $p(\lambda_0,t)$ is of the order $t^{m}$ as $t\to 0$, where $m$ is the geometric multiplicity of $\lambda_0$, or equivalently, $m$ is the number of Jordan blocks. On the other hand, by \eqref{eq: weierstrass charateristic polynomial local form}, $p(\lambda_0,z)$ is of the order $z^{M}$ as $z\to 0$. Hence, $M$ equals $m$.

Next, we show that $\{j_i\}_{i=1,\ldots,M}$ are the sizes of the Jordan blocks. Again, by Weierstrass preparation theorem, the analytic function $g_i(\lambda,z)=z-z_{i}(\lambda)$ in variables $\lambda$ and $z$ has the following local form near $(\lambda_0,0)$:
\begin{equation}\label{eq: weierstrass z minus zjlambda local form}
 z-z_{i}(\lambda)=z^{\ell_i}\left((\lambda-\lambda_0)^{\tilde{j}_i}+c_{i,\tilde{j}_i-1}(z)(\lambda-\lambda_0)^{\tilde{j}_i-1}+\cdots+c_{i,0}(z)\right)b_{i}(\lambda,z),
\end{equation}
where $\ell_i$ and $\tilde{j}_i$ are integers, the analytic function $b_{i}(\lambda,z)$ doesn't vanish near $(\lambda_0,0)$ and the analytic functions $\{c_{i,k}(z)\}_{k=0,\ldots,\tilde{j}_i-1}$ vanish at $0$. Clearly, $\ell_i$ is zero. Otherwise, the set of eigenvalues of $\gamma(0)$ would contain an open neighborhood of $\lambda_0$. Taking $z=0$ and compare with the expansion \eqref{eq: expansion of zjlambda} of $z_i(\lambda)$, we find that $\tilde{j}_i=j_i$. Combining \eqref{eq: weierstrass charateristic polynomial local form}, \eqref{eq: factorization of zM aM-1ZM-1 a0} and \eqref{eq: weierstrass z minus zjlambda local form}, we get that
\begin{equation}
 p(\lambda,z)=\prod_{i=1}^{m}p_{\lambda_0,i}(\lambda,z)\cdot f(\lambda,z),
\end{equation}
where $f(\lambda,z)=b(\lambda,z)\prod_{i=1}^{m}b_{i}(\lambda,z)$ and \[p_{\lambda_0,i}(\lambda,z)=(\lambda-\lambda_0)^{j_i}+c_{i,j_i-1}(z)(\lambda-\lambda_0)^{j_i-1}+\cdots+c_{i,0}(z).\] Taking $z=0$, we see that $\sum_{i=1}^{m}j_i$ equals the algebraic multiplicity of $\lambda_0$. Moreover, for $z$ close to $0$, the roots of $p(\lambda,z)$ near $\lambda_0$ coincide with those of $\prod_{i=1}^{m}p_{\lambda_0,i}(\lambda,z)$ with multiplicities. Comparing \eqref{eq: asymp eigenvalues} with the asymptotics $\lambda-\lambda_0\sim h_{i}'(0)z^{\frac{1}{j_i}}$ of the roots of $\{p_{\lambda_0,i}(\lambda,z)\}_{i=1,\ldots,m}$, we conclude that $\{j_i\}_{i=1}^{m}$ are precisely the sizes of Jordan blocks. This completes the argument for the analyticity of eigenvalues and their first order asymptotics when $t$ varies from $0$.

Next, we prove the part b) of Theorem~\ref{thm: krein lyubarskii c1}. We only present the proof for the case that $t$ increases from $0$. The other case is essentially the same and is left to the reader. Together with the first order asymptotics in \eqref{eq: asymp eigenvalues}, it suffices to show that for $t$ close to $0$,
\begin{itemize}
\item[i)] the eigenvalues moving tangential to the circle actually move along the circle
\item[ii)] they are Krein definite.
\end{itemize}
By Theorem~\ref{thm: krein lyubarskii c1} a), i) implies the semi-simplicity of these eigenvalues on the circle for non-zero real $t$ close to $0$.

If i) fails, then there exist an integer $j$, a real number $v$, an analytic function $h$ and a sequence $(t_n,\lambda_n)$ such that $t_n$ decreases to $0$ as $n$ increases to infinity, $\lambda_n\notin U$, $\lambda_n$ is an eigenvalue of $\gamma(t_n)$, $\lambda_n-\lambda_0=h(t_n^{\frac{1}{j}})$ and $\lambda_n-\lambda_0\overset{n\to\infty}{\sim}\sqrt{-1}\lambda_0\cdot v\cdot t_n^{\frac{1}{j}}$. For each $n$, let us consider the eigenvalues $\lambda_0+h(t_n^{\frac{1}{j}}e^{\sqrt{-1}\varphi_n/j})$ of $\gamma(t_ne^{\sqrt{-1}\varphi_n})$. As $\varphi_n$ increases from $-\pi$ to $\pi$, they rotate around $\lambda_0$ for roughly $\frac{2\pi}{j}$ radians. By first order estimates of the eigenvalues, for $n$ sufficient large, there exists $\phi_n\notin \pi\mathbb{Z}$ such that $\gamma(t_ne^{\sqrt{-1}\phi_n})$ has an eigenvalue on $U$. (Indeed, $\phi_n\to 0$ as $n\to\infty$.) This contradicts with Lemma~\ref{lem: eig on U real z} since $t_ne^{\sqrt{-1}\phi_n}\notin\mathbb{R}$.

Next, we show that the eigenvalues moving on the circle are Krein definite when $t$ is sufficiently close to $0$ with their Krein types determined by their moving directions.

Let us verify the statement as $t$ increases from $0$. The other case is similar and we left the proof to the reader. We have seen that the eigenvalue $\lambda(t)=\lambda_0+h(t^{\frac{1}{j}})$ for certain integer $j$ and certain analytic function $h$. Note that $\lambda'(t)=\frac{1}{j}h'(t^{\frac{1}{j}})t^{\frac{1}{j}-1}$. By continuity of $h'$ and $h'(0)\neq 0$, we see that the eigenvalue on $U$ has a deterministic moving direction along $U$ as $t$ increases from $0$ a bit. If the eigenvalue on $U$ situates on the counter-clockwise direction of $\lambda_0$ in the local sense, then the eigenvalue moves counter-clockwise along the circle as $t$ slightly increases from $0$. Hence, by Theorem~\ref{thm: krein lyubarskii c1} a), there is no Krein indefinite eigenvalue on $U$ situating on the counter-clockwise side of $\lambda_0$. Together with their moving direction, by Theorem~\ref{thm: krein lyubarskii c1} a), we see that those eigenvalues must be semi-simple and Krein positive definite. (Otherwise, if there exists $t_1>0$ such that one of those eigenvalues on the circle is Krein indefinite, then according to the branching mechanism described in Theorem~\ref{thm: krein lyubarskii c1} a) together with the fact that no eigenvalue entrances or escapes $U$ during this period of time, there must exist eigenvalues with different moving directions on the counter-clockwise side of $\lambda_0$ on $U$, which is a contradiction.) Similarly, if an eigenvalues on $U$ situates on the clockwise side of $\lambda_0$, then it is Krein negative definite and moves clockwise along the circle.

Finally, we will see that the eigenvalues of $\gamma(z)$ branching from $\lambda_0$ are semi-simple for $z$ in a small enough punctured disk of $0$. Moreover, the corresponding eigenvectors are also multi-valued analytic functions and admit Puiseux expansion. To see this, it suffices to prove that there exist $m$ $\mathbb{C}^{2n}$-valued analytic functions $\{v_i\}_{i=1,\ldots,m}$ such that $\{v_i(z^{\frac{1}{j_i}})\}_{i=1,\ldots,m}$ is a set of $\sum_{i=1}^{m}j_i$ many linearly independent vectors for sufficiently small $z$ and for $i=1,\ldots,m$, we have that
\begin{equation}\label{eq: analytic eig value eig vector h v}
 \gamma(z^{j_i})v_{i}(z)=(h_{i}(z)+\lambda_0)v_{i}(z).
\end{equation}
Define a family of operators analytic in $z$:
\[T_i(z)\overset{\text{def}}{=}(\gamma(z^{j_i})^{T}-(\bar{\lambda_0}-\overline{h_{i}(\bar{z})})\cdot\Id)(\gamma(z^{j_i})-(\lambda_0-h_{i}(z))\cdot\Id).\]
Note that $T_i(z)^{*}=T_i(z)$ for real valued $z$. Such a family of operator is said to be \emph{symmetric}. By perturbation theories of symmetric operators \cite[Sections~6.1 and 6.2, Chapter~2]{KatoMR1335452}, eigenvalues and corresponding eigenvectors of $T_i$ are analytical for real $z$. More precisely, there exist $m$ analytic complex-valued functions $\mu_{i,1}, \mu_{i,2},\ldots,\mu_{i,m}$ and $m$ analytic $\mathbb{C}^{2n}$-valued functions $\zeta_{i,1},\ldots,\zeta_{i,m}$ such that $\zeta_{i,1},\ldots,\zeta_{i,m}$ are orthonormal and $T_i(t)\zeta_{i,k}(t)=\mu_{i,k}(t)\zeta_{i,k}(t)$ for $k=1,\ldots,m$ and real $t$ close to $0$. Since non-zero analytic functions have isolated zeros, there exist $\delta>0$ and an integer $g:=g(i)=g(h_i)$ such that $\mu_{i,1},\ldots,\mu_{i,g}$ are identically zero and $\mu_{i,g+1},\ldots,\mu_{i,m}$ are non-zero on $[-\delta,0)\cup(0,\delta]$. Note that $T_i(t)\zeta_{i,k}(t)=0$ iff
\begin{equation}\label{eq: analytic eig vector zeta}
 \gamma(t^{j_i})\zeta_{i,k}(t)=(h_{i}(t)+\lambda_0)\zeta_{i,k}(t).
\end{equation}
Hence, for real and sufficiently small $t$, $g(i)$ equals to the geometric multiplicity of the eigenvalue $h_{i}(t)+\lambda_0$ of the matrix $\gamma(t^{j_i})$.

We define an equivalence relation on the set $\{1,\ldots,m\}$: $i\sim i'$ if either $i=i'$ or $j_i=j_{i'}$ and $h_i(z)=h_{i'}(ze^{2k(i,i')\pi\sqrt{-1}/j_i})$ for some integer $k(i,i')$. Then, $i\sim i'$ iff $h_i(z^{\frac{1}{j_i}})$ and $h_{i'}(z^{\frac{1}{j_{i'}}})$ are the same multi-valued analytic functions. In particular, $i\sim i'$ implies that $j_i=j_{i'}$.

Let us firstly consider a special (yet generic) case that the equivalence relation ``$\sim$'' coincides with the standard one ``$=$''. Since non-zero analytic functions have isolated zeros, for different $i$ and $i'$, $h_i(z^{\frac{1}{j_i}})$ and $h_i'(z^{\frac{1}{j_i}})$ are disjoint in a punctured neighbourhood of $0$. In this case, together with the first order asymptotics in \eqref{eq: asymp eigenvalues}, we see that the eigenvalues of $\gamma(z)$ branching from $\lambda_0$ have algebraic multiplicity $1$ as $z$ varies from $0$. For $i=1,\ldots,m$, we take $v_i$ to be the direct analytic continuation of $\zeta_{i,1}$. Then, they satisfy \eqref{eq: analytic eig value eig vector h v}. Moreover, for $z$ in a small enough punctured neighborhood of $0$, the set $\{v_i(z^{\frac{1}{j_i}})\}_{i=1,\ldots,m}$ is linearly independent since they are the eigenvectors of different eigenvalues of $\gamma(z)$.

The general case is more complicated. For the set of eigenvectors $\{v_i(z^{\frac{1}{j_i}})\}_{i=1,\ldots,m}$, we wish to take all $\zeta_{i,k}(z^{\frac{1}{j_i}})$ for $i=1,\ldots,m$ and $k\leq g(i)$. However, there exist duplications: if $i\sim \tilde{i}$, then $h_i(z^{\frac{1}{j_i}})$ and $h_{i'}(z^{\frac{1}{j_{i'}}})$ are the same multi-valued analytic functions and hence, the linear spaces $\Span\{\zeta_{i,k}(z^{\frac{1}{j_i}}):k\leq g(i)\}$ and $\Span\{\zeta_{i',k}(z^{\frac{1}{j_{i'}}}):k\leq g(i')\}$ are identical. Instead of collecting vectors $\zeta_{i,k}(z^{\frac{1}{j_i}})$ for each $i=1,\ldots,m$, we collect vectors for each equivalence class $[i]$, where we denote by $[i]$ the equivalence class of $i$ with respect to the equivalence relation $\sim$. We will show that for $i=1,\ldots,m$, $g(i)$ equals the cardinality $\sharp [i]$ of $[i]$ so that we have the correct number of eigenvectors, i.e., $\sum_{[i]}g(i)j_i=\sum_{i=1}^{m}j_i$. Clearly, $g(i)\leq \sharp [i]$ since the algebraic multiplicity dominates the geometric multiplicity. To get the converse inequality, recall that the eigenvalues branching from $\lambda_0$ are semi-simple for real $t$ close to $0$, which implies that $g(i)=\sharp [i]$ if $h_i(t)\in U$ for real $t$ close to $0$. To obtain the inequality in the general case, we may perform a rotation $t\mapsto te^{2\ell\pi\sqrt{-1}/j_i}$ in \eqref{eq: analytic eig vector zeta} for some properly chosen integer $\ell$. Eventually, for $i=1,\ldots,m$, we define $v_i$ in the following way.
\begin{itemize}
 \item[1)] Take the equivalence class $[i]$ of $i$ and list the integers in $[i]$ in increasing order.
 \item[2)] Find the smallest element $\ell(i)$ in $[i]$ and define $k(i)=\sharp \{i'\in[i]:i'\leq i\}$.
 \item[3)] Define $v_i$ to be the direct analytic continuation of $\zeta_{\ell(i),k(i)}$.
\end{itemize}
The linear independence of the set of vectors $\{v_i(z^{\frac{1}{j_i}})\}_{i=1,\ldots,m}$ is left to the reader.
\bibliographystyle{amsalpha}

\begin{thebibliography}{Win10}

\bibitem[Eke90]{EkelandMR1051888}
Ivar Ekeland, \emph{Convexity methods in {H}amiltonian mechanics}, Ergebnisse
  der Mathematik und ihrer Grenzgebiete (3) [Results in Mathematics and Related
  Areas (3)], vol.~19, Springer-Verlag, Berlin, 1990. \MR{1051888}

\bibitem[GfL58]{GelfandLidskiiMR0091390}
I.~M. Gel\cprime~fand and V.~B. Lidski\u\i, \emph{On the structure of the
  regions of stability of linear canonical systems of differential equations
  with periodic coefficients}, Amer. Math. Soc. Transl. (2) \textbf{8} (1958),
  143--181. \MR{0091390}

\bibitem[Kat95]{KatoMR1335452}
Tosio Kato, \emph{Perturbation theory for linear operators}, Classics in
  Mathematics, Springer-Verlag, Berlin, 1995, Reprint of the 1980 edition.
  \MR{1335452}

\bibitem[KL62]{KreinLyubarskiiMR0142832}
M.~G. Kre\u\i{n} and G.~Ja. Ljubarski\u\i, \emph{Analytic properties of the
  multipliers of periodic canonical differential systems of positive type},
  Izv. Akad. Nauk SSSR Ser. Mat. \textbf{26} (1962), 549--572. \MR{0142832}

\bibitem[Kre50]{KreinMR0036379}
M.~G. Kre\u\i{n}, \emph{A generalization of some investigations of {A}. {M}.
  {L}yapunov on linear differential equations with periodic coefficients},
  Doklady Akad. Nauk SSSR (N.S.) \textbf{73} (1950), 445--448. \MR{0036379}

\bibitem[Kre51]{KreinMR0043980}
\bysame, \emph{On certain problems on the maximum and minimum of characteristic
  values and on the {L}yapunov zones of stability}, Akad. Nauk SSSR. Prikl.
  Mat. Meh. \textbf{15} (1951), 323--348. \MR{0043980}

\bibitem[KY06]{KuwamuraYanagidaMR2271499}
Masataka Kuwamura and Eiji Yanagida, \emph{Krein's formula for indefinite
  multipliers in linear periodic {H}amiltonian systems}, J. Differential
  Equations \textbf{230} (2006), no.~2, 446--464. \MR{2271499}

\bibitem[Lon99]{LongMR1674313}
Yiming Long, \emph{Bott formula of the {M}aslov-type index theory}, Pacific J.
  Math. \textbf{187} (1999), no.~1, 113--149. \MR{1674313}

\bibitem[Lon02]{LongMR1898560}
\bysame, \emph{Index theory for symplectic paths with applications}, Progress
  in Mathematics, vol. 207, Birkh\"auser Verlag, Basel, 2002. \MR{1898560}

\bibitem[Mos58]{MoserMR0096872}
J\"urgen Moser, \emph{New aspects in the theory of stability of {H}amiltonian
  systems}, Comm. Pure Appl. Math. \textbf{11} (1958), 81--114. \MR{0096872}

\bibitem[Win10]{Winitzki2010}
Sergei Winitzki, \emph{Linear algebra via exterior products}, lulu.com, 2010.

\bibitem[YS75]{YakubovichStarzhinskiiMR0364740}
V.~A. Yakubovich and V.~M. Starzhinskii, \emph{Linear differential equations
  with periodic coefficients. 1, 2}, Halsted Press [John Wiley \& Sons]\ New
  York-Toronto, Ont.,; Israel Program for Scientific Translations,
  Jerusalem-London, 1975, Translated from Russian by D. Louvish. \MR{0364740}

\end{thebibliography}
\providecommand{\bysame}{\leavevmode\hbox to3em{\hrulefill}\thinspace}
\providecommand{\MR}{\relax\ifhmode\unskip\space\fi MR }
\providecommand{\MRhref}[2]{%
  \href{http://www.ams.org/mathscinet-getitem?mr=#1}{#2}
}
\providecommand{\href}[2]{#2}

\end{document}